\documentclass{amsart}
\usepackage{amssymb}
\usepackage{amsmath}
\usepackage{amsfonts}

\newtheorem{theorem}{Theorem}
\theoremstyle{plain}

\newtheorem{corollary}{Corollary}

\newtheorem{lemma}{Lemma}

\newtheorem{proposition}{Proposition}
\newtheorem{remark}{Remark}

\numberwithin{equation}{section}
\input{tcilatex}

\begin{document}
\title[Skew group algebras and invariants]{Skew group algebras, invariants\\
and Weyl Algebras}
\author{Roberto Mart\'{\i}nez-Villa}
\address[A. One and A. Two]{Centro de Ciencias Matem\'{a}ticas, UNAM\\
Morelia, Mich. M\'{e}xico}
\email[A. One]{mvilla@matmor.unam.mx}
\urladdr{http://www.matmor.unam.mx}
\thanks{}
\author{Jeronimo Mondrag\'{o}n}
\curraddr[A. Two]{Centro de Ciencias Matem\'{a}ticas, UNAM\\
Morelia, Mich. M\'{e}xico}
\email[A.~Two]{jeronimo@matmor.unam.mx}
\urladdr{http://www.matmor.unam.mx}
\thanks{}
\date{November 5, 2012}
\dedicatory{}
\thanks{}

\begin{abstract}
The aim of this paper is two fold:

First to study finite groups $G$ of automorphisms of the homogenized Weyl
algebra $B_{n}$, the skew group algebra $B_{n}\ast G$, the ring of
invariants $B_{n}^{G}$, and the relations of these algebras with the Weyl
algebra $A_{n}$, with the skew group algebra $A_{n}\ast G$, and with the
ring of invariants $A_{n}^{G}$. Of particular interest is the case $n=1$.

In the on the other hand, we consider the invariant ring $\QTR{sl}{C}[X]^{G}$
of the polynomial ring $K[X]$ in $n$ generators, where $G$ is a finite
subgroup of $Gl(n,\QTR{sl}{C}$) such that any element in $G$ different from
the identity does not have one as an eigenvalue. We study the relations
between the category of finitely generated modules over $\QTR{sl}{C}[X]^{G}$
and the corresponding category over the skew group algebra $\QTR{sl}{C}%
[X]\ast G$. We obtain a generalization of known results for $n=2$ and $G$ a
finite subgroup of $Sl(2,$\textsl{C }$)$. In the last part of the paper we
extend the results for the polynomial algebra $C[X]$ to the homogenized Weyl
algebra $B_{n}$.
\end{abstract}

\maketitle

\section{Automorphism of the homogenized Weyl algebra.}

\bigskip In tis section we will assume the reader is familiar with basic
results on Weyl algebras as in $[Co]$ , For results on the homogenized Weyl
algebra we refer to [MMo]

Let $K$ be a field of zero characteristic. In this section we consider the
it homogenized Weyl algebra $B_{n}$ defined by generators and relations as:

$B_{n}=K<X_{1},X_{2},...X_{n},Y_{1},Y_{2},...Y_{n},Z>/\{[X_{i},\delta
_{j}]-\partial _{ij}Z^{2}$, $[X_{i},X_{j}]$, $[Y_{i},Y_{j}]$, $[X_{i},Z`]$, $%
[Y_{i},Z`]\}$, with $K<X_{1},X_{2},...X_{n},Y_{1},Y_{2},...Y_{n},Z>$the free
algebra in $2n+1$ generators, $[u,v`]$ the commutator $uv-vu$ and $\delta
_{ij}$, Kronecker's delta.

\bigskip It is known $B_{n}$ has a Poincare-Birkoff basis and it is
quadratic, hence by [Li], [GH] it is Koszul. Let $B_{n}^{!}$ be its Yoneda
algebra [GM1],[GM2], $B_{n}^{!}=\underset{k\geq 0}{\oplus }Ext_{B}^{k}(K,K)$%
. The algebra $B_{n}^{!}$ has the same quiver as $B_{n}$ and relations
orthogonal with respect to the canonical bilinear form, it is easy to see
that $B_{n}^{!}$ has the following form: $%
B_{n}^{!}=K_{q}[X_{1},X_{2},...X_{n},Y_{1},Y_{2},...Y_{n},Z]/%
\{X_{i}^{2},Y_{j}^{2},\underset{i=0}{\overset{n}{\tsum }X_{i}Y_{i}+Z^{2}\}%
\text{, }}$\newline
where $K_{q}[X_{1},X_{2},...X_{n},Y_{1},Y_{2},...Y_{n},Z]$ denote the
quantum polynomial ring.

$K<X_{1},X_{2},...X_{n},Y_{1},Y_{2},...Y_{n},Z>/%
\{(X_{i},X_{j}),(Y_{i},Y_{j}),(X_{i},Z),(Y_{i},Z)\}$. Here $(u,v)$ denotes
the anti commutators $uv+vu$.

The polynomial algebra $C_{n}=K[X_{1},X_{2},...X_{n},Y_{1},Y_{2},...Y_{n},Z]$
is a Koszul algebra with Yoneda algebra $%
C_{n}^{!}=K_{q}[X_{1},X_{2},...X_{n},Y_{1},Y_{2},...Y_{n},Z]/%
\{X_{i}^{2},Y_{j}^{2}\}$. The Weyl algebra is defined by $%
A_{n}=K<X_{1},X_{2},...X_{n},Y_{1},Y_{2},...Y_{n},>/\{[X_{i},Y_{j}]-\partial
_{ij},[X_{i},X_{j}]`,[Y_{i},Y_{j}],[X_{i},Z`],[Y_{i},Z]`\}$.

Observe we obtain $C_{n}$ as a quotient of $B_{n}$ and $C_{n}^{!}$ is a sub
algebra of $B_{n}^{!}$.

\bigskip These algebras are related as follows: $B_{n}$ $/ZB_{n}$ $\cong
C_{n}$ and $B_{n}$ $/(Z-1)B_{n}$ $\cong A_{n}$

\bigskip The algebra $B_{n}^{!}$ is a free $C_{n}^{!}$-module of rank two.

In fact we have:

\begin{proposition}
There exists a $C_{n}^{!}$-module decomposition: $B_{n}^{!}=C_{n}^{!}\oplus
ZC_{n}^{!}$.
\end{proposition}

We consider now a finite group $G$ of grade preserving automorphisms of $%
B_{n}$.

It was proved in [MMo $]$ that the center of $B_{n}$ is $K[Z]$ and any
automorphism $\sigma $ $\in G$ takes the center to the center, and since $%
\sigma $ is grade preserving $\sigma (Z)$ is an homogeneous element of
degree one, hence $\sigma (Z)=\lambda _{\sigma }Z$. We are assuming $G$ has
finite order $m$, it follows $\sigma ^{m}(Z)=Z=\lambda _{\sigma }^{m}Z$ ,
and $\lambda _{\sigma }$ is an $m$-th root of unity. If we denote by $\sqrt[m%
]{1}$ the group of $m$ roots of unity, there is a group homomorphism $\eta
:G\rightarrow \sqrt[m]{1}$ given by $\eta (\sigma )=\lambda _{\sigma }$ ,
since $\sqrt[m]{1}$ is a cyclic group the image of $\eta $ is cyclic and we
have a group extension: $1\rightarrow N\rightarrow G\rightarrow \QTR{sl}{Z}%
_{k}\rightarrow 0$ with $\QTR{sl}{Z}_{k}$ the cyclic group of order $k$ and $%
N$ is the subgroup of $G$ such that for every $\sigma $ in $G$, $\sigma
(Z)=Z $.

Since $B_{n}$ is generated in degree one, the action of $G$ on $B_{n}$ is
determined by the action on $M=(B_{n})_{1}=\overset{n}{\underset{i=1}{\oplus 
}}KX_{i}\oplus (\overset{n}{\underset{i=1}{\oplus }}KY_{i})\oplus KZ$.

To determine the structure of $G$ we need to look for automorphisms of $M$
which leave $KZ$ invariant and preserve the relations: $[X_{i},\delta _{j}]$=%
$\partial _{ij}Z^{2}$, $[X_{i},X_{j}]$= $[Y_{i},Y_{j}]$= 0 .

For any element $\sigma $of $G$ we have equations:

\begin{center}
$\sigma ($X$_{j})=\underset{i=1}{\overset{n}{\sum }}$A$_{2i-1,2j-1}$X$_{i}+%
\underset{i=1}{\overset{n}{\sum }}$A$_{2i,2j-1}$Y$_{i}+\mu _{j}Z$

$\sigma ($Y$_{k})=\underset{\ell =1}{\overset{n}{\sum }}$A$_{2\ell -1,2k}$X$%
_{\ell }+\underset{\ell =1}{\overset{n}{\sum }}$A$_{2\ell ,2k}$Y$_{\ell
}+\nu _{k}Z$

$\sigma (Z)=\lambda Z$.
\end{center}

We must have equalities:

$\sigma ($X$_{j}$X$_{k}-$X$_{k}$X$_{j})=0=\sigma ($X$_{j})\sigma ($X$%
_{k})-\sigma ($X$_{k})\sigma ($X$_{j})$.

Using the relations that define $B_{n}$, we obtain after cancellation the
following equation involving $2\times 2$ determinants:

\begin{center}
$\overset{n}{\underset{i=1}{\dsum }}\left\vert 
\begin{array}{cc}
A_{2i-1,2j-1} & A_{2i,2j-1} \\ 
A_{2i-1,2k-1} & A_{2i,2k-1}%
\end{array}%
\right\vert $(X$_{i}$Y$_{i}-$Y$_{i}$X$_{i}$)= 0.
\end{center}

Since for every $i$ we have X$_{i}$Y$_{i}-$Y$_{i}$X$_{i}$=Z$^{2}$, it
follows:

\begin{center}
$\overset{n}{\underset{i=1}{\dsum }}\left\vert 
\begin{array}{cc}
A_{2i-1,2j-1} & A_{2i,2j-1} \\ 
A_{2i-1,2k-1} & A_{2i,2k-1}%
\end{array}%
\right\vert $=0.
\end{center}

We may assume $j<k$. Using that a matrix and its transpose have the same
determinant we obtain the equivalent relations:

\begin{center}
1) $\overset{n}{\underset{i=1}{\dsum }}\left\vert 
\begin{array}{cc}
A_{2i-1,2j-1} & A_{2i-1,2k-1} \\ 
A_{2i,2j-1} & A_{2i,2k-1}%
\end{array}%
\right\vert $=0, for all $j$, $k$ with $j<k$.
\end{center}

\bigskip In a similar way, for we obtain from the relation:

\begin{center}
$\sigma ($Y$_{j}$Y$_{k}-$Y$_{k}$Y$_{j})=0=\sigma ($Y$_{j})\sigma ($Y$%
_{k})-\sigma ($Y$_{k})\sigma ($Y$_{j})$, the following equation:

2) $\overset{n}{\underset{i=1}{\dsum }}\left\vert 
\begin{array}{cc}
A_{2i-1,2j} & A_{2i-1,2k} \\ 
A_{2i,2j} & A_{2i,2k}%
\end{array}%
\right\vert $=0, for all $j$, $k$ with $j<k$.
\end{center}

From the relation:

\begin{center}
$\sigma ($X$_{j}$Y$_{j}-$Y$_{j}$X$_{j})=\sigma ($X$_{j})\sigma ($Y$%
_{j})-\sigma ($Y$_{j})\sigma ($X$_{j})=\sigma (Z^{2})=\lambda ^{2}Z^{2}$,
the following equation:

$\overset{n}{\underset{i=1}{\dsum }}\left\vert 
\begin{array}{cc}
A_{2i-1,2j-1} & A_{2i-1,2j} \\ 
A_{2i,2j-1} & A_{2i,2j}%
\end{array}%
\right\vert $(X$_{i}$Y$_{i}-$Y$_{i}$X$_{i}$)=$\lambda ^{2}$ Z$^{2}$.
\end{center}

Since for every $i$ we have X$_{i}$Y$_{i}-$Y$_{i}$X$_{i}$=Z$^{2}$, it
follows:

\begin{center}
3) $\overset{n}{\underset{i=1}{\dsum }}\left\vert 
\begin{array}{cc}
A_{2i-1,2j-1} & A_{2i-1,2j} \\ 
A_{2i,2j-1} & A_{2i,2j}%
\end{array}%
\right\vert =\lambda ^{2}$
\end{center}

With similar calculations, from the equation:

\begin{center}
$\sigma ($X$_{j}$Y$_{k}-$Y$_{k}$X$_{j})=0=\sigma ($X$_{j})\sigma ($Y$%
_{k})-\sigma ($Y$_{k})\sigma ($X$_{j})$, for $k\neq 0$.
\end{center}

we get the equation:

\begin{center}
$\overset{n}{\underset{i=1}{\dsum }}\left\vert 
\begin{array}{cc}
A_{2i-1,2j-1} & A_{2i-1,2k} \\ 
A_{2i,2j-1} & A_{2i,2k}%
\end{array}%
\right\vert ($X$_{i}$Y$_{i}-$Y$_{i}$X$_{i})=0.$
\end{center}

which implies:

\begin{center}
4) $\overset{n}{\underset{i=1}{\dsum }}\left\vert 
\begin{array}{cc}
A_{2i-1,2j-1} & A_{2i-1,2k} \\ 
A_{2i,2j-1} & A_{2i,2k}%
\end{array}%
\right\vert =0$ , for $i\neq k.$
\end{center}

We have proved the following:

\begin{theorem}
Let $B_{n}$ be the homogenized Weyl algebra in n+1 generators. Then an
automorphism $\sigma $ of $M=\overset{n}{\underset{i=1}{\oplus }}%
KX_{i}\oplus (\overset{n}{\underset{i=1}{\oplus }}KY_{i})\oplus KZ$, with
matrix in block form:$\left[ 
\begin{array}{cc}
A & 0 \\ 
\rho & \lambda%
\end{array}%
,\right] $ where $\rho $ is the vector: $\rho =(\mu _{1},\mu _{2},...\mu
_{n},\nu _{1},\nu _{2}...\nu _{n})$ and $\lambda \neq 0$, extends to an
automorphism of $B_{n}$ if and only if $A$ satisfies the following equations:

1) $\overset{n}{\underset{i=1}{\dsum }}\left\vert 
\begin{array}{cc}
A_{2i-1,2j-1} & A_{2i-1,2k-1} \\ 
A_{2i,2j-1} & A_{2i,2k-1}%
\end{array}%
\right\vert $=0, for all $j$, $k$ with $j<k,$

2) $\overset{n}{\underset{i=1}{\dsum }}\left\vert 
\begin{array}{cc}
A_{2i-1,2j} & A_{2i-1,2k} \\ 
A_{2i,2j} & A_{2i,2k}%
\end{array}%
\right\vert $=0, for all $j$, $k$ with $j<k$.

3) $\overset{n}{\underset{i=1}{\dsum }}\left\vert 
\begin{array}{cc}
A_{2i-1,2j-1} & A_{2i-1,2j} \\ 
A_{2i,2j-1} & A_{2i,2j}%
\end{array}%
\right\vert =\lambda ^{2},$4) $\overset{n}{\underset{i=1}{\dsum }}\left\vert 
\begin{array}{cc}
A_{2i-1,2j-1} & A_{2i-1,2k} \\ 
A_{2i,2j-1} & A_{2i,2k}%
\end{array}%
\right\vert =0$ , for $i\neq k.$
\end{theorem}

A particular case is obtained when $\rho =0$, $\lambda =1$ and the matrix $A$
satisfies $\left\vert 
\begin{array}{cc}
A_{2i-1,2i-1} & A_{2i-1,2i} \\ 
A_{2i,2i-1} & A_{2i,2i}%
\end{array}%
\right\vert =1$ for all $i$ and all remaining $2\times 2$ minors of $A$ are
zero. This is the product of $n$ matrices of seize $2\times 2$ and
determinant one.

\begin{corollary}
Let $G_{1},G_{2},...G_{n}$ be finite subgroups of $Sl(2,K),$then the product 
$G=$ $G_{1}\times G_{2}\times ...\times G_{n}$ acts as automorphism group of
the homogenized algebra in $n+1$ variables, $B_{n}.$
\end{corollary}

\section{Structure of the homogenized Weyl algebra $B_{n}$ \newline
and its skew group algebra $B_{n}\ast G$.}

In this section we study the structure of the homogenized Weyl algebras $%
B_{n}$. We will see that they can be obtained from the homogenized algebras $%
B_{i}$, $B_{j}$ with $i+j=n$, as follows: $B_{n}=B_{i}\otimes
_{K}B_{j}/(Z\otimes 1-1\otimes Z)B_{i}\otimes _{K}B_{j}$. This result is
very similar to the situation of the Weyl algebras for which it is well
known $[$Co$]$ that for $i+j=n$, there is an isomorphism of $K$-algebras: $%
A_{n}=A_{i}\otimes _{K}A_{j}$ or the polynomial algebras $%
C_{i}=K[X_{1},X_{2},...X_{i}]$ , $C_{j}=K[X_{i+1},X_{i+2},...X_{n}]$ for
which the isomorphism $C_{n}\cong C_{i}\otimes C_{n-i}$ is well known.

\begin{theorem}
For integers $n$, $m>0$ let $B_{n}$, $B_{m}$ and $B_{n+m}$ be the
corresponding homogenized algebras. Then there exists an isomorphism of
graded $K$-algebras:

$B_{n+m}=B_{n}\otimes _{K}B_{m}/(Z\otimes 1-1\otimes Z)B_{n}\otimes
_{K}B_{m} $.
\end{theorem}

\begin{proof}
The algebras $B_{n}$, $B_{m}$ and $B_{n+m}$ can be written as:

$B_{n}=K<X_{1},X_{2},...X_{n},Y_{1},Y_{2},...Y_{n},Z>/\{[X_{i},\delta
_{j}]-\partial _{ij}Z^{2}$, $[X_{i},X_{j}]$, $[Y_{i},Y_{j}]$, $[X_{i},Z`]$, $%
[Y_{i},Z`]\}$

$B_{m}=K<X_{n+1},X_{n+2},...X_{n+m},Y_{n+1},Y_{n+2},...Y_{n+m},Z>/\{[X_{i},%
\delta _{j}]-\partial _{ij}Z^{2}$, $[X_{i},X_{j}]$, $[Y_{i},Y_{j}]$, $%
[X_{i},Z`]$, $[Y_{i},Z`]\}$

$B_{n+m}=K<X_{1},X_{2},...X_{n+m},Y_{1},Y_{2},...Y_{n+m},Z>/\{[X_{i},\delta
_{j}]-\partial _{ij}Z^{2}$, $[X_{i},X_{j}]$, $[Y_{i},Y_{j}]$, $[X_{i},Z`]$, $%
[Y_{i},Z`]\}$.

Then there are natural inclusions: $\varphi _{1}:B_{n}\rightarrow B_{n+m}$
and $\varphi _{2}:B_{m}\rightarrow B_{n+m}$ given by: $\varphi
_{i}(X_{j})=X_{j}$, $\varphi _{i}(Y_{j})=Y_{j}$, $\varphi _{i}(Z)=Z$ and $%
i=1,2$.

Let $j_{1}:B_{n}\rightarrow B_{n+m}$ and $j_{2}:B_{m}\rightarrow B_{n+m}$,
be the inclusions $j_{1}(b_{1})=b_{1}\otimes 1$, and $j_{2}(b_{2})=1\otimes
b_{2}$. Since multiplication is bilinear we have a vector space map: $%
\varphi :B_{n}\otimes _{K}B_{m}\rightarrow B_{n+m}$ given by $\varphi
(b_{1}\otimes b_{2})=\varphi _{1}(b_{1})\varphi _{2}(b_{2})$, such that the
diagram:

*) $%
\begin{array}{ccccc}
B_{n} & \overset{j_{1}}{\rightarrow } & B_{n}\otimes _{K}B_{m} & \overset{%
j_{2}}{\longleftarrow } & B_{m} \\ 
& \varphi _{1}\searrow & \downarrow \varphi & \swarrow \varphi _{2} &  \\ 
&  & B_{n+m} &  & 
\end{array}%
$

commutes.

Since $\varphi _{1}(b_{1})\varphi _{2}(b_{2})=\varphi _{2}(b_{2})\varphi
_{1}(b_{1})$ with $b_{1}\in B_{n}$, $b_{2}\in B_{m}$ , $\varphi $ is an
algebra homomorphism. It is clear $\varphi $ is surjective and $\varphi
(Z\otimes 1-1\otimes Z)=0$, hence $(Z\otimes 1-1\otimes Z)B_{n}\otimes
_{K}B_{m}$ is contained in the kernel of $\varphi $. We will prove that they
are actually equal.

Let $b$ be an element of degree $t$ in the kernel of $\varphi $.

The element $b$ has the following form:

$\underset{_{\alpha +\beta +k+,\ell +\alpha ^{\prime }+\beta ^{\prime }=t}}{%
\dsum }C_{\alpha ,\beta ,k,}X^{\alpha }Y^{\beta }Z^{k}\otimes B_{\alpha
^{\prime },\beta ^{\prime },\ell ,}X_{m}^{\alpha ^{\prime }}Y_{m}^{\beta
^{\prime }}Z^{\ell }$, with $X^{\alpha }=X_{1}^{\alpha _{1}}X_{2}^{\alpha
_{2}}...X_{n}^{\alpha _{n}}$, $Y^{\beta }=Y_{1}^{\beta _{1}}Y_{2}^{\beta
_{2}}...Y_{n}^{\beta _{n}}$ and $X_{m}^{\alpha ^{\prime }}=X_{n+1}^{\alpha
_{1}^{\prime }}X_{n+2}^{\alpha _{2}^{\prime }}...X_{n+m}^{\alpha
_{m}^{\prime }}$, $Y_{m}^{\beta ^{\prime }}=Y_{n+1}^{\beta _{1}^{\prime
}}Y_{n+2}^{\beta _{2}^{\prime }}...Y_{n+m}^{\beta _{m}^{\prime }}$.

Assume $k>0.$

Then $X^{\alpha }Y^{\beta }Z^{k}\otimes X_{m}^{\alpha ^{\prime
}}Y_{m}^{\beta ^{\prime }}Z^{\ell }=X^{\alpha }Y^{\beta }Z^{k-1}\otimes
X_{m}^{\alpha ^{\prime }}Y_{m}^{\beta ^{\prime }}Z^{\ell }(Z\otimes
1)=X^{\alpha }Y^{\beta }Z^{k-1}\otimes X_{m}^{\alpha ^{\prime }}Y_{m}^{\beta
^{\prime }}Z^{\ell }(Z\otimes 1-1\otimes Z)+X^{\alpha }Y^{\beta
}Z^{k-1}\otimes X_{m}^{\alpha ^{\prime }}Y_{m}^{\beta ^{\prime }}Z^{\ell +1}$%
.

By induction,

$C_{\alpha ,\beta ,k,}X^{\alpha }Y^{\beta }Z^{k}\otimes B_{\alpha ^{\prime
},\beta ^{\prime },\ell ,}X_{m}^{\alpha ^{\prime }}Y_{m}^{\beta ^{\prime
}}Z^{\ell }$=$g(X,Y,Z)(Z\otimes 1-1\otimes Z)$+$C_{\alpha ,\beta
,k,}X^{\alpha }Y^{\beta }\otimes B_{\alpha ^{\prime },\beta ^{\prime },\ell
,}X_{m}^{\alpha ^{\prime }}Y_{m}^{\beta ^{\prime }}Z^{k+\ell }$ and $%
g(X,Y,Z) $ an expression in $X$, $Y$, $Z$..

Then

$\underset{_{\alpha +\beta +k+,\ell +\alpha ^{\prime }+\beta ^{\prime }=t}}{%
\dsum }C_{\alpha ,\beta ,k,}X^{\alpha }Y^{\beta }Z^{k}\otimes B_{\alpha
^{\prime },\beta ^{\prime },\ell ,}X_{m}^{\alpha ^{\prime }}Y_{m}^{\beta
^{\prime }}Z^{\ell }=G(X,Y,Z)$ $(Z\otimes 1-1\otimes Z)$+$\underset{r}{\dsum
(}\underset{r=k+\ell }{\underset{_{\alpha +\beta +\alpha ^{\prime }+\beta
^{\prime }=t-r}}{\dsum }}C_{\alpha ,\beta ,k}B_{\alpha ^{\prime },\beta
^{\prime },\ell }X^{\alpha }Y^{\beta }Z^{k}\otimes X_{m}^{\alpha ^{\prime
}}Y_{m}^{\beta ^{\prime }})(1\otimes Z^{r}).$

Applying $\varphi $ to the above expression we get:

$\varphi (b)=0=\underset{r}{\dsum (}\underset{r=k+\ell }{\underset{_{\alpha
+\beta +\alpha ^{\prime }+\beta ^{\prime }=t-r}}{\dsum }}C_{\alpha ,\beta
,k}B_{\alpha ^{\prime },\beta ^{\prime },\ell }X_{1}^{\alpha
_{1}}X_{2}^{\alpha _{2}}...X_{n}^{\alpha _{n}}Y_{1}^{\beta _{1}}Y_{2}^{\beta
_{2}}...Y_{n}^{\beta _{n}}$\newline
$X_{n+1}^{\alpha _{1}^{\prime }}X_{n+2}^{\alpha _{2}^{\prime
}}...X_{n+m}^{\alpha _{m}^{\prime }}Y_{n+1}^{\beta _{1}^{\prime
}}Y_{n+2}^{\beta _{2}^{\prime }}...Y_{n+m}^{\beta _{m}^{\prime }})Z^{r}.$

Using the fact $X_{n+j}Y_{i}=Y_{i}X_{n+j}$ for $1\leqslant i\leqslant m$ we
obtain:

$\underset{r}{\dsum (}\underset{r=k+\ell }{\underset{_{\alpha +\beta +\alpha
^{\prime }+\beta ^{\prime }=t-r}}{\dsum }}C_{\alpha ,\beta ,k}B_{\alpha
^{\prime },\beta ^{\prime },\ell }X_{1}^{\alpha _{1}}X_{2}^{\alpha
_{2}}...X_{n}^{\alpha _{n}}X_{n+1}^{\alpha _{1}^{\prime }}X_{n+2}^{\alpha
_{2}^{\prime }}...X_{n+m}^{\alpha _{m}^{\prime }}$

$Y_{1}^{\beta _{1}}Y_{2}^{\beta _{2}}...Y_{n}^{\beta _{n}}Y_{n+1}^{\beta
_{1}^{\prime }}Y_{n+2}^{\beta _{2}^{\prime }}...Y_{n+m}^{\beta _{m}^{\prime
}})Z^{r}=0.$

It follows:

$\underset{r=k+\ell }{\underset{_{\alpha +\beta +\alpha ^{\prime }+\beta
^{\prime }=t-r}}{\dsum }}C_{\alpha ,\beta ,k}B_{\alpha ^{\prime },\beta
^{\prime },\ell }X^{\alpha }Y^{\beta }\otimes X_{m}^{\alpha ^{\prime
}}Y_{m}^{\beta ^{\prime }}=0.$

Therefore:

$\underset{_{\alpha +\beta +k+,\ell +\alpha ^{\prime }+\beta ^{\prime }=t}}{%
\dsum }C_{\alpha ,\beta ,k,}X^{\alpha }Y^{\beta }Z^{k}\otimes B_{\alpha
^{\prime },\beta ^{\prime },\ell ,}X_{m}^{\alpha ^{\prime }}Y_{m}^{\beta
^{\prime }}Z^{\ell }=G(X,Y,Z)$ $(Z\otimes 1-1\otimes Z)$ is an element of $%
(Z\otimes 1-1\otimes Z)B_{n}\otimes _{K}B_{m}$ , as claimed.
\end{proof}

Let $\Lambda $, $\Gamma $ be $K$-algebras and $G$, $H$ finite group of
automorphisms of $\Lambda $, $\Gamma $, respectively. Assume the
characteristic of $K$ does not divide neither the order of $G$ nor the order
of $H$. By the universal property of the coproduct, given $\sigma \in G$, $%
\tau \in H$, there is a commutative diagram:

\begin{center}
$%
\begin{array}{ccccc}
\Lambda & \overset{j_{1}}{\rightarrow } & \Lambda \otimes _{K}\Gamma & 
\overset{j_{2}}{\longleftarrow } & \Gamma \\ 
\sigma \downarrow &  & (\sigma ,\tau )\downarrow &  & \downarrow \tau \\ 
\Lambda & \overset{j_{1}}{\rightarrow } & \Lambda \otimes _{K}\Gamma & 
\overset{j_{2}}{\longleftarrow } & \Gamma%
\end{array}%
$
\end{center}

with $(\sigma ,\tau )(\lambda \otimes \gamma )=\sigma (\lambda )\otimes \tau
(\gamma )$, the homomorphism $(\sigma ,\tau )$ is an automorphism, such that 
$(\sigma ,\tau )=1$ implies $\sigma =1$ and $\tau =1.$Hence $G\times H$
embeds faithfully in the automorphism group of $\Lambda \otimes _{K}\Gamma .$

The following proposition is well known: (see for example $[CR]$).

\begin{proposition}
Let $K$ be a field and $G$, $H$, finite groups. Then there is a natural
isomorphism of finite dimensional algebras: $K(G\times H)\cong KG\otimes
_{K}K(H)$.
\end{proposition}

\begin{proof}
We define a map $\varphi :K(G\times H)\rightarrow KG\otimes _{K}K(H)$ as $%
\varphi (g,h)=g\otimes h$ and extend it linearly. The map sends basis to
basis, hence is a vector space isomorphism. It is easy to see it is also a
ring isomorphism.
\end{proof}

We use this result to prove the following:

\begin{theorem}
Let $\Lambda $, $\Gamma $ be graded $K$-algebras and $G$, $H$ finite group
of (grade preserving) automorphisms of $\Lambda $, $\Gamma $, respectively.
Assume the characteristic of $K$ does not divide neither the order of $G$
nor the order of $H$. Denote by $\Lambda \ast G$ and $\Gamma \ast H$ the
corresponding skew group algebras. Then there is a natural isomorphism of
(graded) $K$-algebras: $\Lambda \otimes _{K}\Gamma \ast (G\times H)\cong
\Lambda \ast G\otimes _{K}\Gamma \ast H$.
\end{theorem}

\begin{proof}
We have vector space isomorphisms: $\Lambda \otimes _{K}\Gamma \ast (G\times
H)\cong \Lambda \otimes _{K}\Gamma \otimes _{K}K(G\times H)\cong \Lambda
\otimes _{K}\Gamma \otimes _{K}KG\otimes _{K}K(H)\cong \Lambda \otimes
_{K}KG\otimes _{K}\Gamma \otimes _{K}K(H)\cong \Lambda \ast G\otimes
_{K}\Gamma \ast H$.

We look to the action of $G\times H$ on $\Lambda \otimes _{K}\Gamma .$

Let $(\sigma ,\tau )$ be an element of $G\times H$ and $\lambda \otimes
\gamma \in $ $\Lambda \otimes _{K}\Gamma .$Then ($\sigma ,\tau )\lambda
\otimes \gamma =\lambda ^{\sigma }\sigma \otimes \gamma ^{\tau }\tau
=\lambda ^{\sigma }\otimes \gamma ^{\tau }$($\sigma ,\tau )=\lambda ^{\sigma
}\otimes \gamma ^{\tau }\otimes \sigma \otimes \tau =\lambda ^{\sigma
}\otimes \sigma \otimes \gamma ^{\tau }\otimes \tau =$ ($\sigma \otimes \tau
)(\lambda \otimes \gamma ).$

It follows $\Lambda \otimes _{K}\Gamma \ast (G\times H)\cong \Lambda \ast
G\otimes _{K}\Gamma \ast H$ as (graded) algebras.
\end{proof}

\begin{corollary}
Let $K$- be a field of zero characteristic, $B_{n}$, $B_{m}$ homogenized
Weyl algebras and $G,$ $H$, finite groups of grade preserving automorphisms
of $B_{n}$ and $B_{m},$respectively. Assume for all $\sigma \in G$ and $\tau
\in H$, $\sigma (Z)=Z$ and $\tau (Z)=Z$. Then $G\times H$ acts on $B_{n+m}$
and there is an isomorphism of graded algebras: $B_{n+m}\ast (G\times
H)\cong B_{n}\ast G\otimes _{K}B_{m}\ast H/(Z\otimes 1-1\otimes Z)B_{n}\ast
G\otimes _{K}B_{m}\ast H$.

If we denote by $B_{n+m}^{G\times H}$, $B_{n}^{G}$, $B_{m}^{H}$ the rings of
invariants, then we have an isomorphism of algebras: $B_{n+m}^{G\times
H}\cong B_{n}^{G}\otimes _{K}B_{m}^{H}$ $/Z\otimes 1$- $1\otimes
Z)B_{n}^{G}\otimes _{K}B_{m}^{H}$.
\end{corollary}

\begin{proof}
Given $\sigma \in G$, $\tau \in H$, we have a commutative diagram:

$%
\begin{array}{ccccc}
B_{n} & \overset{j_{1}}{\rightarrow } & B_{n}\otimes _{K}B_{m} & \overset{%
j_{2}}{\longleftarrow } & B_{m} \\ 
\sigma \downarrow &  & \sigma \otimes \tau \downarrow &  & \downarrow \tau
\\ 
B_{n} & \overset{j_{1}}{\rightarrow } & B_{n}\otimes _{K}B_{m} & \overset{%
j_{2}}{\longleftarrow } & B_{m} \\ 
& \varphi _{1}\searrow & \downarrow \varphi & \swarrow \varphi _{2} &  \\ 
&  & B_{n+m} &  & 
\end{array}%
$

Let $\underset{i=1}{\overset{n}{\sum }}b_{i}\otimes b_{i}^{\prime }$ be an
element of the kernel of $\varphi \sigma \otimes \tau $. Then $\varphi ($ $%
\underset{i=1}{\overset{n}{\sum }\sigma }b_{i}\otimes \tau b_{i}^{\prime
})=0.$

By that above description of $Ker\varphi $, $\underset{i=1}{\overset{n}{\sum 
}\sigma }b_{i}\otimes \tau b_{i}^{\prime }=g(X,Y,Z)(Z\otimes 1$-$1\otimes Z)$%
.

Therefore: \newline
$\underset{i=1}{\overset{n}{\sum }}b_{i}\otimes b_{i}^{\prime }$ =(($\sigma
^{-1},\tau ^{-1}$)($g(X,Y,Z))(\sigma ^{-1}Z\otimes 1$-$1\otimes \tau ^{-1}Z)$%
=$g^{\prime }(X,Y,Z)(Z\otimes 1$-$1\otimes Z).$

It follows $Ker\varphi \sigma \otimes \tau $=$(Z\otimes 1$-$1\otimes
Z)B_{n}\otimes _{K}B_{m}$ and the map $\varphi \sigma \otimes \tau $ factors
through $B_{n+m},$denote by $\overline{\sigma \otimes \tau }$ the induced
map, which is clearly an automorphism of $B_{n+m}$ such that the diagram:

$%
\begin{array}{ccccc}
B_{n} & \overset{\varphi _{1}}{\rightarrow } & B_{n+m} & \overset{\varphi
_{2}}{\longleftarrow } & B_{m} \\ 
\sigma \downarrow &  & \overline{\sigma \otimes \tau }\downarrow &  & 
\downarrow \tau \\ 
B_{n} & \overset{\varphi _{1}}{\rightarrow } & B_{n+m} & \overset{\varphi
_{2}}{\longleftarrow } & B_{m}%
\end{array}%
$

The above diagram *) induces a commutative diagram of graded algebras and
homomorphisms:

$%
\begin{array}{ccccc}
B_{n}\ast G & \overset{\overline{j_{1}}}{\rightarrow } & B_{n}\otimes
_{K}B_{m}(G\times H) & \overset{\overline{j_{2}}}{\longleftarrow } & 
B_{m}\ast H \\ 
& \overline{\varphi _{1}}\searrow & \downarrow \overline{\varphi } & 
\swarrow \overline{\varphi _{2}} &  \\ 
&  & B_{n+m}\ast (G\times H) &  & 
\end{array}%
$

There is an exact sequence:

0$\rightarrow (Z\otimes 1$- $1\otimes Z)B_{n}\otimes _{K}B_{m}(G\times
H)\rightarrow B_{n}\otimes _{K}B_{m}(G\times H)\rightarrow B_{n+m}\ast
(G\times H)\rightarrow $0

From the isomorphism ($B_{n}\otimes _{K}B_{m})\ast (G\times H)\cong
B_{n}\ast G\otimes _{K}B_{m}\ast H$, it follows $B_{n+m}\ast (G\times
H)\cong B_{n}\ast G\otimes _{K}B_{m}\ast H/(Z\otimes 1$-$1\otimes
Z)B_{n}\ast G\otimes _{K}B_{m}\ast H$.

Now let $e=1/\left\vert G\right\vert \underset{g\in G}{\sum g}$, and $%
f=1/\left\vert H\right\vert \underset{h\in H}{\sum }h$ be elements of $KG$
and $KH$, respectively. The elements $e$ and $f$ are idempotents, it is well
known and easy to prove that the rings $B_{n}^{G}$ and $B_{m}^{H}$ are
isomorphic $e(B_{n}\ast G)e$ and $f(B_{m}\ast H)f$, respectively.

Under the isomorphism $K(G\times H)\rightarrow KG\otimes _{K}KH$ the
idempotent $(e,f)=1/\left\vert G\times H\right\vert \underset{(g,h)\in
G\times H}{\sum }(g,h)$ correspond to $e\otimes f.$

Then from the isomorphism: ($B_{n}\otimes _{K}B_{m})\ast (G\times H)\cong
B_{n}\ast G\otimes _{K}B_{m}\ast H$ we get an isomorphism:

$(e,f)$($B_{n}\otimes _{K}B_{m})\ast (G\times H)(e,f)\cong e\otimes
f(B_{n}\ast G\otimes _{K}B_{m}\ast H)e\otimes f\cong e(B_{n}\ast G)e\otimes
_{K}f(B_{m}\ast H)f.$

Therefore ($B_{n}\otimes _{K}B_{m})^{G\times H}\cong B_{n}^{G}\otimes
_{K}B_{m}^{H}.$

It follows $(e,f)$($B_{n+m})\ast (G\times H)(e,f)\cong B_{n+m}^{G\times
H}\cong (e\otimes f)(B_{n}\ast G\otimes _{K}B_{m}\ast H)(e\otimes
f)/(Z\otimes 1-1\otimes Z)$ $\cong e(B_{n}\ast G)e\otimes _{K}f(B_{m}\ast
H)f/(Z\otimes 1-1\otimes Z)$.

From this we have the isomorphism of algebras:

$B_{n+m}^{G\times H}\cong B_{n}^{G}\otimes _{K}B_{m}^{H}$ $/(Z\otimes 1$- $%
1\otimes Z)B_{n}^{G}\otimes _{K}B_{m}^{H}$.
\end{proof}

For Weyl algebras we have the following analogous of the previous theorem.

\begin{theorem}
Given finite groups of automorphisms $G$, $H$ of the Weyl algebras $A_{n}$
and $A_{m}$, $G\times H$ acts as a group of automorphisms of $A_{n+m}$ and
there is an isomorphism of $K$-algebras, $A_{n+m}^{G\times H}\cong
A_{n}^{G}\otimes _{K}A_{m}^{H}$, where $A_{n+m}^{G\times H}$, $A_{n}^{G}$, $%
A_{m}^{H}.$
\end{theorem}

\section{The Weyl algebras $B_{1}$, and $A_{1}$ and their skew group
algebras $B_{1}\ast G$ and $A_{1}\ast G$, with $G$ a finite subgroup of $%
Sl(2,C)$.}

In this section we describe the basic algebras Morita equivalent to $%
B_{1}\ast G$ and $A_{1}\ast G$ by quivers and relations. To achieve this we
will make use of the following result, $[$AR$]$, $[$L$]$, $[$C-B$]$:

\begin{theorem}
For any subgroup $G$ of $Sl(2,C)$ the skew group algebra \emph{C}$[X,Y]\ast
G $ is Morita equivalent to the preprojective algebra of an Euclidean
diagram.
\end{theorem}

We will end the section sketching the situation for a product $G=G_{1}\times
G_{2}\times ...G_{n}$ of finite subgroups $G_{i}$ of $Sl(2,C)$ acting on $%
B_{n}$ and on $A_{n}$.

We start by recalling the situation of \emph{C}$[X,Y]\ast G$, following the
approach of $[$GuM$]$ and take the opportunity to correct some inaccuracies.

Recall the construction of a McKay quiver of a finite subgroup $G$ of the
linear group $Gl(n,C)$. [ Mc]

Let $S_{1}$, $S_{2},...S_{n}$ be the non isomorphic irreducible
representations of $G$ and $M$ the representation corresponding to the
inclusion of $G$ in $Gl(n,C)$. Tensoring $M$ with some $S_{j}$ we obtain a
decomposition in irreducible representations: $M\otimes _{K}S_{j}\cong 
\underset{i}{\oplus }a_{ij}S_{i}$.

The McKay quiver of $G$ has vertices $v_{1}$, $v_{2}$,... $v_{n}$, with each
vertex $v_{i}$ corresponding to an irreducible representation $S_{i}$ and we
put $a_{ij}$ arrows from $v_{i}$ to $v_{j}$.

For the proof we will make use of the following well known result from
ordinary group representations:

\begin{theorem}
$[$Mu$]$ Let $G$ be a finite group and $L$ a complex irreducible
representation. Then the dimension of $L$ as $C$-vector space divides the
order of $G$.
\end{theorem}

\bigskip We reproduce here the proof given in $[$Ste$]$ of the following:

\begin{theorem}
Let $M$ be a $C$-vector space of dimension $2$ and $G$ a finite subgroup of
the special linear group $Sl(2,M).$Then the McKay quiver has no loops.
\end{theorem}

\begin{proof}
For the proof of the theorem we have two cases:

1) $M$ is an irreducible representation.

According to the previous theorem, $2$ divides the order of $G$ and by Sylow
theorems, there exists an element $g\in G$ of order $2.$ Since $g\in $ $%
sl(2,M)$, by a change of basis $PgP^{-1}=\left[ 
\begin{array}{cc}
\lambda & 0 \\ 
\mu & \lambda ^{\prime }%
\end{array}%
\right] $, with $\lambda \lambda ^{\prime }=1$ , and $\left[ 
\begin{array}{cc}
\lambda & 0 \\ 
\mu & \lambda ^{\prime }%
\end{array}%
\right] \left[ 
\begin{array}{cc}
\lambda & 0 \\ 
\mu & \lambda ^{\prime }%
\end{array}%
\right] =\left[ 
\begin{array}{cc}
\lambda ^{2} & 0 \\ 
\lambda \mu +\lambda ^{\prime }\mu & \lambda ^{\prime 2}%
\end{array}%
\right] =1$

Therefore: $\lambda ^{2}=1=(\lambda ^{\prime })^{2}$ and either $\lambda
=\lambda ^{\prime }=1$ or $\lambda =\lambda ?=-1$, in any case $\mu =0$ and $%
g$ is in the center of $G.$

Let $S$ be another irreducible representation of $G$ and $\varphi :$ $%
CG\rightarrow S$ an epimorphism.

Let%
%TCIMACRO{\U{b4}}%
%BeginExpansion
\'{}%
%EndExpansion
s suppose $g-1$ is not in the kernel of $\varphi $. Then there exists $%
s^{\prime }\in S$ with $(g-1)s^{\prime }\neq 0$. Since $S$ is simple, for
any $s\in S$, $s\neq 0$, there exists $\rho \in $ $CG$ such that $\rho
(g-1)s^{\prime }=s$ and $\rho (g-1)=(g-1)\rho $. It follows $%
(g+1)s=(g+1)(g-1)\rho s^{\prime }=(g^{2}-1)\rho s^{\prime }=0$ and $gs=-s.$

This means that $g$ acts on $S$ either as the identity or as $-1$.

If $g$ acts as $-1$ on $M$, and as $1$ on $S$, then $g$ acts as -$1$ on $%
M\otimes _{K}S$ and $S$ can not appear as a summand of $M\otimes _{K}S$ and
if $g$ acts as $-1$ on $M$ and as $-1$ on $S$, then it acts as $1$ on $%
M\otimes _{K}S$ and again $S$ can not be a summand of $M\otimes _{K}S$.

We have proved that in this case the McKay quiver has no loops.

2) The representation $M$ is reducible, this is: $M=M_{1}\oplus M_{2}$ and $%
\dim _{C}M_{1}=\dim _{C}M_{2}=1$. Let%
%TCIMACRO{\U{b4}}%
%BeginExpansion
\'{}%
%EndExpansion
s say that $M_{1}$ is generated by $m_{1}$and $M_{2}$ is generated by $m_{2}$%
. If the order of $G$ is $n$, then for any $g\in G$, $g(m_{1})=\lambda
_{1}m_{1}$ and $g(m_{2})=\lambda _{2}m_{2}$ with $\lambda _{1}$and $\lambda
_{2}$ n-th roots of unity.

The element $g$ has form: $g=\left[ 
\begin{array}{cc}
\lambda _{1} & 0 \\ 
0 & \lambda _{2}%
\end{array}%
\right] $, with $\lambda _{1}$. $\lambda _{2}$=$1$.

There exists an injective homomorphism from $G$ to the group of nth roots of
unity, $\sqrt[n]{1}$ given by $g\rightarrow \lambda _{1}$, hence $G$ is
cyclic and all irreducible representations have dimension one.

Let $Cs$ $=S$ be an irreducible representation $G=<g>$ and $gs=ts.$

Suppose $M\otimes _{K}S=S_{1}\oplus S_{2}$. Then $M_{1}\otimes S=S_{1}$ and $%
M_{2}\otimes S=S_{2}$.

Assume $S=S_{1}$, then $m_{1}\otimes s=s^{\prime }$, $s^{\prime }\in S$ and $%
s^{\prime }=rs$. It follows $g(m_{1}\otimes s)=g(m_{1})\otimes g(s)=\lambda
_{1}m_{1}\otimes ts=g(rs)=rts=t\lambda _{1}(m_{1}\otimes s)=t\lambda _{1}rs.$%
Therefore $\lambda _{1}=1$ and $g$ is the identity.

We have proved the McKay quiver has no loops.
\end{proof}

We sketch here the arguments used in [GuM] to prove the theorem.

Let $G$ be a finite subgroup of $Sl(2,C)$, which extends to a group of grade
preserving automorphisms of $C[X,Y]$. The algebra $C[X,Y]$ is Koszul with
Yoneda algebra the exterior algebra $\Lambda =C<X,Y>/\{X^{2},Y^{2}$, $%
XY+YX\} $, it was proved in $[$MV1$]$ that $G$ acts in a natural way as an
automorphism group of the Yoneda algebra, hence $G$ is a sub group of the
automorphisms group of $\Lambda $.

It was also proved in $[$MV1$]$, that the skew group algebra $C[X,Y]\ast G$
is Koszul with Yoneda algebra $\Lambda \ast G.$

The algebra $\Lambda $ is selfinjective of radical cube zero. It follows
from $[$RR$]$, $\Lambda \ast G$ is selfinjective of radical cube zero. Hence 
$C[X,Y]\ast G$ is Artin-Schelter regular of global dimension two [Sm],[MV4].
It is also clear that $C[X,Y]\ast G$ is noetherian.

Consider now an arbitrary Koszul $C$-algebra $\Lambda $ and $G$ a finite
group of automorphisms of $\Lambda $. Given a complete set of primitive
orthogonal idempotents $e_{1}$, $e_{2}$,... $e_{n}$ of $CG$, they are also a
complete set of orthogonal idempotents of $\Lambda \ast G$ and taking $e=%
\overset{n}{\underset{i=1}{\sum }}e_{i}$, the algebra $e(\Lambda \ast G)e$
is basic and Morita equivalent to $\Lambda \ast G$.

There is a natural isomorphism:

\begin{center}
$Ext_{\Lambda \ast G}^{k}(CGe,CGe)\cong eExt_{\Lambda \ast G}^{k}(CG,CG)e$
\end{center}

By the Morita theorems applied to graded algebras [MV4], the functor $%
Hom_{\Lambda \ast G}(\Lambda \ast Ge,-)$ induces an equivalence of
categories $Gr_{\Lambda \ast G}\cong Gr_{e\Lambda \ast Ge}$ .

It follows that for each $k\geq 0$, there is an isomorphism:

\begin{center}
$Ext_{\Lambda \ast G}^{k}(CGe,CGe)\cong Ext_{e\Lambda \ast
Ge}^{k}(eCGe,eCGe).$
\end{center}

Using this two isomorphisms and adding up, we get an isomorphism of graded $%
K $-algebras.

\begin{center}
$e(\underset{k\geq 0}{\oplus }Ext_{\Lambda \ast G}^{k}(CG,CG))e\cong 
\underset{k\geq 0}{\oplus }Ext_{e\Lambda \ast Ge}^{k}(eCGe,eCGe)$
\end{center}

In particular when $\Lambda $ is the exterior algebra in two generators we
have:

\begin{center}
$C[X,Y]\ast G\cong \underset{k\geq 0}{\oplus }Ext_{\Lambda \ast
G}^{k}(CG,CG) $ and $e(C[X,Y]\ast G)e\cong \underset{k\geq 0}{\oplus }%
Ext_{e\Lambda \ast Ge}^{k}(eCGe,eCGe).$
\end{center}

The algebra $e(\Lambda \ast G)e$ is basic Koszul selfinjective of radical
cube zero with Yoneda algebra the basic noetherian algebra $e(C[X,Y]\ast G)e$%
.

It follows from [GMT] that the separated quiver of $e(\Lambda \ast G)e$ is
an Euclidean diagram $Q.$

It is easy to check that the quiver of $e(C[X,Y]\ast G)e$ is the McKay
quiver of $G$, by Theorem ?, this quiver does not have loops. By the
properties of Koszul algebras $e(C[X,Y]\ast G)e$ and $e(\Lambda \ast G)e$
have the same quiver $\overset{\wedge }{Q}$ .

We know by $[GMT]$ $\overset{\wedge }{Q}$ is a translation quiver with
translation $\tau $ the Nakayama permutation, but $soc(\Lambda \ast
G)e_{i}=J^{2}\ast Ge_{i}$ , and since $\Lambda $ has simple socle $J^{2}=K$.
Then $soc(\Lambda \ast G)e_{i}=KGe_{i}=top(\Lambda \ast G)e_{i}$ and $\tau $
is the identity.

The quiver $\overset{\wedge }{Q}$ is the complete quiver of an Euclidean
diagram, this means $\overset{\wedge }{Q}_{0}=Q_{0}$ and $\overset{\wedge }{Q%
}_{1}=Q_{i}\cup Q_{1}^{op}$, with $Q$ an Euclidean diagram. For each arrow $%
\alpha :i\rightarrow j$ in $\overset{\wedge }{Q}$ we have an arrow $\alpha
^{-1}:j\rightarrow i$ in $\overset{\wedge }{Q:}%
\begin{array}{ccccc}
&  & j &  &  \\ 
& \alpha \nearrow &  & \searrow \alpha ^{-1} &  \\ 
i &  &  &  & i \\ 
& \beta \searrow &  & \nearrow \beta ^{-1} &  \\ 
&  & k &  & 
\end{array}%
.$

Since $(\Lambda \ast G)e_{i}$ has simple socle, for any pair of arrows $%
\alpha :i\rightarrow ,\beta :i\rightarrow k$, there is a non zero $c\in K$,
such that $\alpha ^{-1}\alpha =c\beta ^{-1}\beta $. If we assume $Q$ is a
tree then we can change the maps $b:(\Lambda \ast G)e_{i}\rightarrow
(\Lambda \ast G)e_{j}$ corresponding to the arrow $\beta $ by $cb$ and we
get an arrow which we denote again by $\beta $ such that $\alpha ^{-1}\alpha
=\beta ^{-1}\beta $ and we obtain an isomorphism $e(\Lambda \ast G)e\cong K$ 
$\overset{\wedge }{Q}/I$, where $I$ is the ideal generated by relations: $%
\alpha ^{-1}\alpha -\beta ^{-1}\beta $ , $\alpha \alpha ^{-1}-\beta \beta
^{-1}$ and $\alpha \delta $ if $\delta \neq \alpha ^{-1}$, $\delta \alpha $
if $\delta \neq $ $\alpha ^{-1}$.

From this is clear that $e(C[X,Y]\ast G)e\cong K$ $\overset{\wedge }{Q}%
/I^{\bot }$, where $I^{\bot }$ is the ideal generated by mesh relations: $%
\begin{array}{ccccc}
&  & j_{1} &  &  \\ 
& \alpha _{1}\nearrow &  & \searrow \alpha _{1}^{-1} &  \\ 
i & \rightarrow & j_{2} & \rightarrow & i \\ 
& \alpha _{k}\searrow &  & \nearrow \alpha _{k}^{-1} &  \\ 
&  & j_{k} &  & 
\end{array}%
$, this is $\overset{k}{\underset{i=1}{\sum }}$ $\alpha _{i}^{-1}\alpha
_{i}\in I^{\bot }.$

We have proved $e(C[X,Y]\ast G)e$ is isomorphic to the preprojective algebra.

The case $\overset{\sim }{A}_{n}$is the skew group algebra corresponding to
the cyclic group $Z_{n}$ and has to be considered separately, since for $%
\overset{\sim }{A}_{n}$it is not clear that we could choose the arrows in
the algebra $e(\Lambda \ast G)e$ in such a way that for any pair of arrows $%
\alpha ,\beta $ with the same origin, $\alpha ^{-1}\alpha =\beta ^{-1}\beta $%
.

We also need to prove that all preprojective algebras appear in this way.

A full description of the quiver of a preprojective algebra and its
relations with the McKay graph for finite subgroups of $Sl(2,C)$ has
appeared in several papers by authors like: Crawley-Boevey or Lenzing [L],
[C-B].

We now consider finite groups of automorphisms $G$ of the homogenized Weyl
algebra $B_{n}$ such that for all $\sigma \in G$, $\sigma (Z)=Z$ and for all 
$1\leq i\leq n$, $\sigma (X_{i})$, $\sigma (Y_{i})\in \underset{i=1}{\overset%
{n}{\oplus }}CX_{i}\oplus \underset{i=1}{\overset{n}{\oplus }}CY_{i}$.

Given a Koszul algebra $\Lambda $ with Yoneda algebra $\Gamma $ and a finite
group of automorphisms $G$ of $\Lambda $, we recall from $[$MV1$]$ , how the
action transfers to a group of automorphisms of $\Gamma $:

Let $M$ be a $\Lambda $-module and $\sigma \in G$, we define $M^{\sigma }$
as the module with $M^{\sigma }=M$ as vector space and multiplication given
as follows: for $\lambda \in \Lambda $, $m\in M^{\sigma }$we define $\lambda
\ast m=\lambda ^{\sigma }m.$

In case $M$ is a $G$-module, there is an isomorphism: $\varphi _{\sigma
}:M\rightarrow M^{\sigma }$given by $\varphi _{\sigma }(m)=\sigma m.$Then $%
\varphi _{\sigma }(\lambda m)=\sigma (\lambda m)=\lambda ^{\sigma }\sigma
m=\lambda \ast \varphi _{\sigma }(m)$.

Now given an extension $\delta \in Ext_{\Lambda }^{k}(S_{i},S_{j})$, with $%
S_{i},$ $S_{j}$ graded simple modules:

\begin{center}
$\delta :0\rightarrow S_{j}\rightarrow E_{1}\rightarrow E_{2}\rightarrow
...E_{k}\rightarrow S_{i}\rightarrow 0$
\end{center}

we define

\begin{center}
$\sigma (\delta ):0\rightarrow S_{j}^{\sigma }\rightarrow E_{1}^{\sigma
}\rightarrow E_{2}^{\sigma }\rightarrow ...E_{k}^{\sigma }\rightarrow
S_{i}^{\sigma }\rightarrow 0.$
\end{center}

It is clear that the modules S$_{j}^{\sigma }$,S$_{i}^{\sigma }$are again
graded simple and $\sigma $($\delta $)$\in $Ext$_{\Lambda }^{k}$(S$%
_{i}^{\sigma }$,S$_{j}^{\sigma }$).

In this way we have an isomorphism of $C$-algebras: $\sigma :\underset{k\geq
0}{\oplus }Ext_{\Lambda }^{k}(\Lambda /J,\Lambda /J)\rightarrow \underset{%
k\geq 0}{\oplus }Ext_{\Lambda }^{k}(\Lambda /J,\Lambda /J)$ given by $\sigma
(\delta _{1},\delta _{2},...\delta _{n})=(\sigma (\delta _{1}),\sigma
(\delta _{2}),...\sigma (\delta _{n}))$.

In the particular case $\Lambda =B_{n}$ and $\Gamma =B_{n}^{!}$ we want to
see that for any group of automorphisms $G$ of $B_{n}$ such that any element 
$\sigma $of $G$ has matrix form $\sigma =\left[ 
\begin{array}{cc}
A & 0 \\ 
0 & 1%
\end{array}%
\right] $, with $A$ a $n-1\times n-1$ matrix, the action of $G$ on $\left(
B_{n}^{!}\right) _{1}$ is such that every $\sigma \in G$ has matrix form $%
\sigma =\left[ 
\begin{array}{cc}
\ast & 0 \\ 
0 & 1%
\end{array}%
\right] $, with $\ast $ a $n-1\times n-1$ matrix.

The element $Z$ of $B_{n}^{!}$ corresponds to the extension:

\begin{center}
$%
\begin{array}{ccccccc}
0\rightarrow & J & \rightarrow & B_{n} & \rightarrow & C & \rightarrow 0 \\ 
& \downarrow &  & \downarrow &  & \downarrow 1 &  \\ 
0\rightarrow & J/J^{2} & \rightarrow & E & \rightarrow & C & \rightarrow 0
\\ 
& \downarrow &  & \downarrow &  & \downarrow 1 &  \\ 
0\rightarrow & C\overline{Z} & \rightarrow & C[Z]/(Z^{2}) & \rightarrow & C
& \rightarrow 0 \\ 
& \downarrow \cong &  & \downarrow &  & \downarrow 1 &  \\ 
0\rightarrow & C & \rightarrow & C[Z]/(Z^{2}) & \rightarrow & C & 
\rightarrow 0%
\end{array}%
$
\end{center}

The extension $\delta =Z:$ $0\rightarrow C\overset{\overline{z}}{\rightarrow 
}CZ/(Z^{2})\rightarrow C\rightarrow 0$ is such that both $C$ and $CZ/(Z^{2})$
are $G$-modules with trivial action. Then for any $\sigma \in G$, ($%
CZ/(Z^{2}))^{\sigma }=CZ/(Z^{2})$ and $C^{\sigma }=C$ . Moreover $\sigma
(\delta )=\delta $ . We have proved that in $B_{n}^{!}$, $\sigma (Z)=Z$ for
all $\sigma \in G$.

We have an isomorphism: $Ext_{B_{n}}^{1}(C,C)\cong Hom_{B_{n}}(J,C)\cong
Hom_{B_{n}}(J/J^{2},C)$, where $J/J^{2}\cong \underset{i=1}{\overset{n}{%
\oplus }}C\overline{X}_{i}\oplus \underset{i=1}{\overset{n}{\oplus }}C%
\overline{Y}_{i}\oplus C\overline{Z}$.

The element $X_{i}$ of $B_{n}^{!}$ corresponds to the map: $f:J\rightarrow
J/J^{2}\rightarrow C\overline{X}_{i}\overset{\cong }{\rightarrow }K$
applying the homomorphism $\sigma $we get a map: $f^{\sigma }:J^{\sigma
}\rightarrow (J/J^{2})^{\sigma }\rightarrow (C\overline{X}_{i})^{\sigma }%
\overset{\cong }{\rightarrow }C.$

But by hypothesis $(\overline{X}_{i})^{\sigma }=\sum A_{ij}X_{j}+\sum
B_{ij}Y_{j},$hence in the expansion of $f^{\sigma }$ does not appear $Z.$

For $Y_{i}\in B_{n}^{!}$ the situation is similar and the automorphism $%
\sigma $ of ($B_{n}^{!})_{1}$has form $\sigma =\left[ 
\begin{array}{cc}
\left[ \ast \right] & 0 \\ 
0 & 1%
\end{array}%
\right] $, with $\left[ \ast \right] $ a $n-1\times n-1$ matrix.

We want to use the previous remarks to describe $B_{1}^{!}\ast G$ for $G$ a
finite subgroup of $Sl(2,C)$ with action on $B_{1}$ as above.

Let $e_{1}$, $e_{2}$,... $e_{k}$ be a complete set of orthogonal idempotents
of $CG$ and $e=\overset{k}{\underset{i=1}{\sum }}e_{i}$, $B_{1}/ZB_{1}\cong
C_{1}$, the polynomial algebra in two variables.

We saw in previous section $eC_{1}\ast Ge$ is the preprojective algebra and
its Yoneda algebra is the selfinjective radical cube zero algebra $%
eC_{1}^{!}\ast Ge$, with $C_{1}^{!}$ the exterior algebra in two variables.
and for any pair of arrows $\alpha :i\rightarrow j$ , $\beta :i\rightarrow k$
in the quiver of $eC_{1}^{!}\ast Ge$ there exists arrows $\alpha
^{-1}:j\rightarrow i$ and $\beta ^{-1}:k\rightarrow i$ such that $\alpha
^{-1}\alpha =$ $\beta ^{-1}\beta $ and $\delta \alpha =0$ if $\delta $ is an
arrow different from $\alpha ^{-1}.$

We also saw in Section 1, that there is an isomorphism of $C_{n}^{!}$%
-bimodules: $B_{n}^{!}\cong C_{n}^{!}\oplus ZC_{n}^{!}.$

If we assume $G$ acts on $B_{n}$ in such a way that for $\sigma \in G$, $%
\sigma (Z)=Z$ and $\sigma (X_{i}),\sigma (Y_{i})$is contained in the vector
space generated by the $X%
%TCIMACRO{\U{b4}}%
%BeginExpansion
{\acute{}}%
%EndExpansion
s$ and the $Y%
%TCIMACRO{\U{b4}}%
%BeginExpansion
{\acute{}}%
%EndExpansion
s$, then $G$acts on the same way on $B_{n}^{!}$ and we have an isomorphism: $%
B_{n}^{!}\ast G\cong C_{n}^{!}\ast G\oplus ZC_{n}^{!}\ast G$ of $%
C_{n}^{!}\ast G$-bimodules.

which induces an isomorphism: $e(B_{n}^{!}\ast G)e\cong e(C_{n}^{!}\ast
G)e\oplus Ze(C_{n}^{!}\ast G)e$ of $e(C_{n}^{!}\ast G)e$-bimodules. In
particular, for $n=1$, $e(B_{1}^{!}\ast G)e\cong e(C_{1}^{!}\ast G)e\oplus
Ze(C_{1}^{!}\ast G)e$.

The Jacobson radical of $B_{1}^{!}$ is of the form: $J=J_{C_{1}^{1}}\oplus
ZC_{1}^{!}$, with $J_{C_{1}^{1}}$ the Jacobson radical of $C_{1}^{!}$. It
follows the Jacobson radical of $B_{1}^{!}\ast G$ is $J\ast
G=J_{C_{1}^{1}}\ast G\oplus ZC_{1}^{!}\ast G$ and ($J\ast
G)^{2}=J_{C_{1}^{1}}^{2}\ast G+Z(J_{C_{1}^{1}}\ast G)+$ $Z^{2}C_{1}^{!}\ast
G $ . But $Z^{2}=-XY\in J_{C_{1}^{1}}^{2}$, hence ($J\ast
G)^{2}=J_{C_{1}^{1}}^{2}\ast G+Z(J_{C_{1}^{1}}\ast G).$

Therefore: $J\ast G/J^{2}\ast G\cong J/J^{2}\ast G\cong
J_{C_{1}^{1}}/J_{C_{1}^{1}}^{2}\ast G\oplus Z(C_{1}^{!}J_{C_{1}^{1}}\ast
G)\cong J_{C_{1}^{1}}/J_{C_{1}^{1}}^{2}\ast G\oplus Z(CG)$. Then $e($ $J\ast
G/J^{2}\ast G)e\cong e/J_{C_{1}^{1}}/J_{C_{1}^{1}}^{2}\ast G)e\oplus
Ze(CG)e. $ Hence for $1\leq i\leq n$, $e_{i}($ $J\ast G/J^{2}\ast
G)e_{j}\cong e_{i}(J_{C_{1}^{1}}/J_{C_{1}^{1}}^{2}\ast G)e_{j}\oplus
Ze_{i}(CG)e_{j}.$

Since the preprojective algebra has no loops, $%
e_{i}(J_{C_{1}^{1}}/J_{C_{1}^{1}}^{2}\ast G)$ $e_{i}=0$. In the other hand,
for $i\neq j$, $e_{i}(CG)e_{j}=0$ and $e_{i}($ $J\ast G/J^{2}\ast
G)e_{i}=ZCe_{i}.$

We have proved that the quiver of $e(B_{1}^{!}\ast G)e$ has the same
vertices and arrows as a preprojective algebra corresponding to an Euclidean
diagram, and in addition a loop for each vertex.

We shall next find the relations.

Given an idempotent $e_{j}^{\prime }\notin \{e_{1}$, $e_{2}$,... $e_{k}\}$,
there exists one of the idempotents $e_{j}$ and an isomorphism $\varphi
:CGe_{j}^{\prime }\rightarrow CGe_{j}$, in particular $\varphi
(e_{j}^{\prime })=e_{j}^{\prime }\rho e_{j}$ and $\varphi
^{-1}(e_{j})=e_{j}\gamma e_{j}^{\prime }$ with $\rho $, $\gamma \in CG$ .
Then $e_{j}=e_{j}\gamma e_{j}^{\prime }\rho e_{j}$ and $e_{j}^{\prime
}=e_{j}^{\prime }\rho e_{j}\gamma e_{j}^{\prime }$.

Note that $CXY=J_{C_{1}^{1}}^{2}$, and $G$ acts trivially on $XY$.

Let $g\in G$, $g(X)=aX+bY$, $g(Y)=cX+dY.$Hence, $g(XY)=g(X)g(Y)=$ ($aX+bY$)$%
(cX+dY)=caX^{2}+adXY+bcYX+bdY^{2}=adX-bcY=XY$, since $det(g)=1.$In
consequence, $XYe_{i}=e_{i}XY=-e_{i}Z^{2}.$

It follows $XYe_{i}=e_{i}XYe_{i}=\overset{n}{\underset{i=1}{\sum }}%
e_{i}Xe_{j}^{\prime }Ye_{i}.$

Consider paths of the following form:

\begin{center}
$%
\begin{array}{ccccc}
\alpha &  & k &  & \alpha ^{-1} \\ 
& \nearrow &  & \searrow &  \\ 
i &  &  &  & i \\ 
& \searrow &  & \nearrow &  \\ 
e_{j}\gamma e_{i}^{\prime }Ye_{i} &  & j &  & e_{i}Xe_{j}^{\prime }\rho e_{j}%
\end{array}%
$
\end{center}

Where $e_{i}Xe_{j}^{\prime }Ye_{i}=e_{i}Xe_{j}^{\prime }\rho e_{j}\gamma
e_{i}^{\prime }Ye_{i}$.

Since each indecomposable projective of $e(B_{1}^{!}\ast G)e$ has simple
socle and $e(B_{1}^{!}\ast G)e$ is a radical cube zero algebra, there are
constants $c_{j}\in C$ such that $c_{j}\alpha ^{-1}\alpha =$ $%
e_{i}Xe_{j}^{\prime }Ye_{i}=e_{i}Xe_{j}^{\prime }\rho e_{j}\gamma
e_{i}^{\prime }Ye_{i}$.

Therefore: $e_{i}XY=(\underset{j=1}{\overset{n}{\sum }}c_{j})\alpha
^{-1}\alpha =t_{i}\alpha ^{-1}\alpha =-e_{i}Z^{2}$, with $t_{i}\in C-\{0\}.$
This is $t_{i}\alpha ^{-1}\alpha +Z_{i}^{2}=0$, with $Z_{i}=Ze_{i}$.

In a similar way we obtain relations $t_{k}\alpha \alpha ^{-1}+Z_{k}^{2}=0$, 
$\alpha ^{-1}\alpha $ and $\alpha \alpha ^{-1}$are paths in $e(C_{1}^{!}\ast
G)e$ .

The element $Z$ of $B_{1}^{!}$ anticommutes with $X$, Y and commutes with
all the elements $g$ of $G.$ Then $Z$ anticommutes with $\alpha $ and $%
\alpha ^{-1}$. This is $Z_{k}\alpha =-\alpha Z_{i}.$

It follows: $0=t_{k}\alpha \alpha ^{-1}\alpha +Z_{k}^{2}\alpha =$ $%
t_{i}\alpha \alpha ^{-1}\alpha +\alpha Z_{i}^{2}$, then $\alpha \alpha
^{-1}\alpha +(Z_{k}^{2}/t_{k})\alpha =\alpha \alpha ^{-1}\alpha +\alpha
Z_{i}^{2}/t_{i}.$

We have the following equalities: $\alpha Z_{i}^{2}/t_{i}=(\alpha
Z)Z_{i}/t_{i}=-Z\alpha Z/t_{i}=(Z^{2}\alpha )/t_{i}=(Z_{k}^{2}/t_{k})\alpha
=(Z^{2}\alpha )/t_{k}$. Therefore $t_{k}=t_{i}$ for ani pair of vertices $i$%
, $k$ connected by an arrow. By connectivity, $t_{k}=t_{i}$ for all $k$, $i$.

Assume every element of $K$ has a square root and let $c=\sqrt[2]{t_{k}}$,
we can make a change of variables, $z_{k}=Z/c$ to get relations: \{$%
z_{k}\alpha +\alpha z_{i}$, for every arrow $\alpha $, and $\alpha \alpha
^{-1}+z_{i}^{2}$, $\alpha ^{-1}\alpha +z_{k}^{2}\}=R$

We have proved the algebra $e(B_{1}^{!}\ast G)e$ is isomorphic to the quiver
algebra $K\overset{\wedge }{Q}$ / $R$ , with $\overset{\wedge }{Q}$ $%
_{0}=Q_{0}$ and $\overset{\wedge }{Q}$ $_{1}=Q_{i}\cup Q_{1}^{op}\cup
\{z_{i}\mid i\in Q_{0}\}$ and $Q$ an Euclidean diagram.

Using Koszulity, $e(B_{1}\ast G)e$ is isomorphic to the quiver algebra $K%
\overset{\wedge }{Q}$ / $R^{\bot }$ , with $\overset{\wedge }{Q}$ $%
_{0}=Q_{0} $ and $\overset{\wedge }{Q}$ $_{1}=Q_{i}\cup Q_{1}^{op}\cup
\{z_{i}\mid i\in Q_{0}\}$ and $Q$ an Euclidean diagram and $R^{\bot }$ the
orthogonal relations: $R^{\bot }=\{z_{k}\alpha -\alpha z_{i}$, $\sum \alpha
_{j}\alpha _{j}^{-1}-z_{i}^{2}$, $\sum \alpha _{j}^{-1}\alpha
_{j}-z_{k}^{2}\}$.

For the first Weyl algebra we have: $e(A_{1}\ast G)e=e(B_{1}\ast
G)e/(z-1)e(B_{1}\ast G)e.$

It follows $A_{1}\ast G$ is Morita equivalent to an algebra with quiver $%
\overset{\wedge }{Q}$, such that $\overset{\wedge }{Q}$ $_{0}=Q_{0}$ and $%
\overset{\wedge }{Q}$ $_{1}=Q_{i}\cup Q_{1}^{op}$ and $Q$ an Euclidean
diagram and relations \{$\sum \alpha _{j}\alpha _{j}^{-1}-e_{k}$, $\sum
\alpha _{j}^{-1}\alpha _{j}-e_{i}$\}. This means, $e(A_{1}\ast G)e$ is the
deformed preprojective algebra.

We state the previous results as a theorem:

\begin{theorem}
Let $G$ be a finite subgroup of $Sl(2,K),$with $K$ a field of zero
characteristic and containing the square root of each element. Let $A_{1}$
be the first Weyl algebra and $B_{1}$ its homogenized algebra. The group $G$
acts as a grade preserving automorphism group of $B_{1}$, fixing $Z$ and
sending $X$ and $Y$ to a linear combination of $X$ and $Y$. Then $G$ acts in
the same way on $B_{1}^{!}$, the Yoneda algebra of $B_{1}$, and $G$ acts as
a group of automorphisms of $A_{1}$. The skew group algebras $B_{1}\ast G$, $%
B_{1}^{!}\ast G$ and $A_{1}\ast G$ are Morita equivalent to the algebras
defined by quivers and relations as follows:

i) The skew group algebra $B_{1}^{!}\ast G$ is Morita equivalent to the
quiver algebra $K\overset{\wedge }{Q}$ / $R$, with $\overset{\wedge }{Q}$ $%
_{0}=Q_{0}$ and $\overset{\wedge }{Q}$ $_{1}=Q_{i}\cup Q_{1}^{op}\cup
\{z_{i}\mid v_{i}\in Q_{0}\}$ and $Q$ an Euclidean diagram and $R=$ \{$%
z_{k}\alpha +\alpha z_{i}$, for every arrow $\alpha $, and $\alpha \alpha
^{-1}+z_{i}^{2}$, $\alpha ^{-1}\alpha +z_{k}^{2}\}$ .

ii) The skew group algebra $B_{1}\ast G$ is Morita equivalent to the quiver
algebra $K\overset{\wedge }{Q}$ / $R^{\bot }$ , with $\overset{\wedge }{Q}$ $%
_{0}=Q_{0}$ and $\overset{\wedge }{Q}$ $_{1}=Q_{i}\cup Q_{1}^{op}\cup
\{z_{i}\mid v_{i}\in Q_{0}\}$ and $Q$ an Euclidean diagram and $R^{\bot }$
the orthogonal relations: $R^{\bot }=\{z_{k}\alpha -\alpha z_{i}$, $\sum
\alpha _{j}\alpha _{j}^{-1}-z_{i}^{2}$, $\sum \alpha _{j}^{-1}\alpha
_{j}-z_{k}^{2}\}$.

iii) The skew group algebra $A_{1}\ast G$ is Morita equivalent to the
algebra $K$ $\overset{\wedge }{Q}/I$, such that $\overset{\wedge }{Q}$ $%
_{0}=Q_{0}$ and $\overset{\wedge }{Q}$ $_{1}=Q_{i}\cup Q_{1}^{op}$ and $Q$
an Euclidean diagram and relations $I=\{\sum \alpha _{j}\alpha
_{j}^{-1}-e_{k}$, $\sum \alpha _{j}^{-1}\alpha _{j}-e_{i}\}$.
\end{theorem}

We want to sketch the situation for groups of the form $G=G_{1}\times
G_{2}\times ...\times G_{n}$, such that $G_{i}$ is a finite subgroup of $%
Sl(2,C)$ acting as subgroups of the automorphism group of the homogenized
Weyl algebra $B_{n}$.

We will recall first the description by quivers and relations of the tensor
product of two quiver algebras.

Let $K$ be a field of zero characteristic and let $\Lambda =KQ_{1}/I_{1}$
and $\Gamma =KQ_{2}/I_{2}$ be two graded quiver algebras, $Q$ the quiver of $%
\Lambda \otimes _{K}\Gamma $ and $I$ an admissible ideal of $KQ$ such that $%
KQ/I\cong $ $\Lambda \otimes _{K}\Gamma $. We next describe $Q$ and $I$.

Let \{ $e_{1}$, $e_{2}$,... $e_{n}$\} and \{ $f_{1}$, $f_{2}$,... $f_{m}$\}
be complete sets of orthogonal primitive idempotents of $\Lambda $ and $%
\Gamma $, respectively. Then \{$e_{i}\otimes f_{j}\mid 1\leq i\leq n$, $%
1\leq j\leq m$\} is a complete set of primitive orthogonal idempotents of $%
\Lambda \otimes _{K}\Gamma .$ This means that the quiver $Q$ has vertices $%
Q_{0}=(Q_{1})_{0}\times (Q_{2})_{0}$, where the pair $(v_{i},w_{j})$
corresponds to the idempotent $e_{i}\otimes f_{j}$ and $v_{i}$ is the vertex
corresponding to $e_{i}$ and $w_{j}$ the vertex corresponding to $f_{j}$.

The arrows $\alpha _{k}:v_{i}\rightarrow v_{j}$ in $(Q_{1})_{1}$ correspond
to elements $a_{k}\in e_{j}r_{\Delta }e_{i}-e_{j}r_{\Delta }^{2}e_{i}$ and
the arrows $\beta _{k}:w_{i}\rightarrow w_{j}$ of $(Q_{2})_{1}$ correspond
to elements $b_{k}\in f_{j}r_{\Gamma }f_{i}-f_{j}r_{\Gamma }^{2}f_{i}$. We
have in $Q$ arrows $Q_{1}=(Q_{1})_{1}\times (Q_{2})_{0}\cup (Q_{1})_{0}\cup
(Q_{2})_{1}$. This is to an arrow $\alpha :v_{i}\rightarrow v_{j}$ in $%
(Q_{1})_{1}$ and a vertex $w_{k}$ of $(Q_{2})_{0}$ corresponds an arrow $%
(\alpha $, $w_{k}):(v_{i},w_{k})\rightarrow $ $(v_{j},w_{k})$ of $Q_{1}$
associated to $a\otimes f_{k}\in e_{i}\otimes f_{k}(r_{\Lambda \otimes
_{K}\Gamma })e_{j}\otimes f_{k}-e_{i}\otimes f_{k}(r_{\Lambda \otimes
_{K}\Gamma }^{2})e_{j}\otimes f_{k}$. Similarly, to an arrow $\beta
:w_{i}\rightarrow w_{j}$ of $(Q_{2})_{1}$ and a vertex $v_{k}$ of $%
(Q_{1})_{0}$ corresponds the arrow $(v_{k},\beta ):(v_{k},w_{i})\rightarrow $
$(v_{k},w_{j})$ of $Q_{1}$ associated to $e_{k}\otimes b\in e_{k}\otimes
f_{i}(r_{\Lambda \otimes _{K}\Gamma })e_{k}\otimes f_{j}-e_{k}\otimes
f_{i}(r_{\Lambda \otimes _{K}\Gamma }^{2})e_{k}\otimes f_{j}$.

Given paths $\alpha _{1}\alpha _{2}...\alpha _{r}:v_{i}\rightarrow v_{j}$
and $\beta _{1}\beta _{2}...\beta _{s}:w_{k}\rightarrow w_{\ell }$ in $Q_{1}$
and $Q_{2}$, respectively, there are paths ($\alpha _{1}\alpha _{2}...\alpha
_{r},w)$: $(v_{i},w)\rightarrow (v_{j},w)$ and $(v,\beta _{1}\beta
_{2}...\beta _{s}):(v,w_{k})\rightarrow (v,w_{\ell })$. Now it is clear that
given readable relations $\rho =\sum c_{k}\gamma _{k}$ in $I_{1}$, and $\rho
^{\prime }=\sum b_{k}\gamma _{k}^{\prime }$ in $I_{2}$ there are induced
readable relations $(\rho ,w)=\sum c_{k}(\gamma _{k},w)$ and $(v,\rho 
%TCIMACRO{\U{b4}}%
%BeginExpansion
{\acute{}}%
%EndExpansion
)=\sum b_{k}(v,\gamma _{k}^{\prime })$ in $KQ$.

If we have two arrows $\alpha :v_{2}\rightarrow v_{1}$ and $\beta
:w_{2}\rightarrow w_{1}$ in $Q_{1}$ and $Q_{2}$, respectively, the we have
maps:

\begin{center}
$%
\begin{array}{ccc}
\Lambda \otimes _{K}\Gamma e_{1}\otimes f_{1} & \overset{a\otimes f_{1}}{%
\rightarrow } & \Lambda \otimes _{K}\Gamma e_{2}\otimes f_{1} \\ 
\downarrow e_{1}\otimes b &  & \downarrow e_{2}\otimes b \\ 
\Lambda \otimes _{K}\Gamma e_{2}\otimes f_{1} & \overset{a\otimes f_{2}}{%
\rightarrow } & \Lambda \otimes _{K}\Gamma e_{2}\otimes f_{2}%
\end{array}%
$
\end{center}

making the diagram commute.

Hence we have in $I$ a commutative relation $\zeta _{\alpha ,\beta
}=(v_{2},\beta )(\alpha _{1},w_{1})-(\alpha _{1},w_{2})(v_{1},\beta )$.

Then $I$ is the ideal generated by the relations \{ $\{\rho \}\times
(Q_{2})_{0}\mid \rho $ a readable relation in $I_{1}$\}$\cup $\{ $%
(Q_{1})_{0}\times \{\rho ^{\prime }\}\mid \rho ^{\prime }$a readable
relation in $I_{2}$\}$\cup \{$ $\zeta _{\alpha ,\beta }\mid \alpha \in
(Q_{1})_{1}$, $\beta \in (Q_{2})_{1}$\}.

We obtained the description of $Q$ and $I$ such that $KQ/I\cong \Lambda
\otimes _{K}\Gamma $.

We return to the case $G=G_{1}\times G_{2}$ with $G_{1}$, $G_{2}$ finite
subgroups of $Sl(2,C)$ acting in the way above described as an automorphism
group of $B_{2}.$

Taking complete sets of primitive orthogonal idempotents \{ $e_{1}$, $e_{2}$%
,... $e_{k}$\} and \{ $f_{1}$, $f_{2}$,... $f_{\ell }$\} of $KG_{1}$ and $%
KG_{2}$, respectively, Letting $e$, $f$ be the idempotents $e=\underset{i=1}{%
\overset{k}{\sum }}e_{i}$ and $f=\overset{\ell }{\underset{i=1}{\sum }}f_{i}$
we saw above $e(B_{1}\ast G_{1})e$ and $f(B_{1}\ast G_{2})f$ are
preprojective algebras of Euclidean diagrams, $e(B_{1}\ast G_{1})e$ $\cong K$
$\overset{\wedge }{Q}_{1}/I_{i}$ and $f(B_{1}\ast G_{2})f$ $\cong K$ $%
\overset{\wedge }{Q}_{2}/I_{2}$ the quiver $\overset{\wedge }{Q}_{1}$is of
the form $(\overset{\wedge }{Q}_{1})_{0}=(Q_{1})_{0}$, $(\overset{\wedge }{Q}%
_{1})_{1}=(Q_{1})_{1}\cup (Q_{1}^{op})_{1}$ $\cup \{z_{i}\mid v_{i}\in
(Q_{1})_{0}\}$ and $(\overset{\wedge }{Q}_{2})_{0}=(Q_{2})_{0}$, $(\overset{%
\wedge }{Q}_{2})_{1}=(Q_{2})_{1}\cup (Q_{2}^{op})_{1}\cup \{z_{i}\mid
w_{i}\in (Q_{2})_{0}\}$ with $Q_{1}$ and $Q_{2}$ Euclidean diagrams. and $%
I_{1}$, $I_{2}$ the ideals described above

Then $e\otimes f(B_{1}\ast G_{1}\otimes _{K}B_{1}\ast G_{2})e\otimes f=$ $%
e(B_{1}\ast G_{1})e\otimes _{K}f(B_{1}\ast G_{2})f$ $\cong K$ $\overset{%
\wedge }{Q}_{1}/I_{i}\otimes K$ $\overset{\wedge }{Q}_{2}/I_{2}$ is the
tensor product of two algebras related to preprojective algebras of
Euclidean diagrams, as described above.

We have $1\otimes Z=\underset{i,j}{\sum }e_{i}\otimes Zf_{j}$ and $Z\otimes
1=\underset{i,j}{\sum }Ze_{i}\otimes f_{j}$, hence $(Z\otimes 1-1\otimes Z)$
is the ideal generated by readable relations $e_{i}\otimes
Zf_{j}-Ze_{i}\otimes f_{j}$.

Therefore in ($e,f)(B_{2}\ast (G_{1}\times G_{2}))(e,f)\cong (e\otimes
f(B_{1}\ast G_{1}\otimes _{K}B_{1}\ast G_{2})e\otimes f)/(Z\otimes
1-1\otimes Z)$ we are identifying the two loops $e_{i}\otimes Zf_{j}$ and $%
Ze_{i}\otimes f_{j}$.

Then $B_{2}\ast (G_{1}\times G_{2})$ can be completely described by quiver
and relations: it has the same quiver as the tensor product of two
preprojective algebras of Euclidean diagrams and in addition a loop in each
vertex, the relations are naturally induced from those in each preprojective
algebra as studied in detail above. We leave to the reader to write
explicitly the quiver and relations as well as the induced relations in $%
(e,f)A_{2}(G_{1}\times G_{2})(e,f)\cong ((e,f)(B_{2}\ast (G_{1}\times
G_{2}))(e,f))/(Z-1).$ Observe that from this could also give a full
description of the relations in $B_{2}^{G_{1}\times G_{2}}$ and $%
A_{2}^{G_{1}\times G_{2}}$.

By induction a full description by quivers and relations of the basic
algebras Morita equivalent to $B_{n}\ast (G_{1}\times G_{2}\times ...G_{n})$
can be given, the algebras come from iterated tensor products of
preprojective algebras of Euclidean diagrams and additional loops for each
vertex and naturally induced relations.

\section{The algebras $C[X]\ast G$ and $C[X]^{G}$, with $G$ a finite
subgroup of $Gl(n,C)$, such that no $\protect\sigma \in G$, $\protect\sigma %
\neq 1$, has a fixed point.}

In this section we study the relations between the categories of finitely
generated modules $\func{mod}_{C[X]\ast G}$ and $\func{mod}_{C[X]^{G}}$,
where $C[X]\ast G$ is the skew group algebra and $C[X]^{G}$ the invariant
ring of a finite subgroup of $Gl(n,C)$, such that no $\sigma \in G$, $\sigma
\neq 1$, has a fixed point.

More precisely, we prove that after killing the modules of finite length,
the categories $\func{mod}_{C[X]\ast G}$ and $\func{mod}_{C[X]^{G}}$ are
equivalent. These results generalize known results for $C[X,Y]$ and finite
subgroups of $Sl(2,C)$, [C-B].

\begin{lemma}
Assume $G$ is a finite subgroup of $Gl(n,C)$. Then $G$ acts naturally on the
polynomial ring in $n$ variables $C[X]$ and the invariant ring $C[X]^{G}$ is
the center of $C[X]\ast G.$
\end{lemma}

\begin{proof}
It is clear $C[X]^{G}$ is contained in the center of $C[X]\ast G$. Let $v\in
Z(C[X]\ast G)$ be an element in the center, $v=\underset{g\in G}{\sum }%
c_{g}g $ and $c_{g}\in C[X]$.

For any $X_{i}v=vX_{i}$ for any $X_{i}$. Then $X_{i}v=\underset{g\in G}{\sum 
}c_{g}X_{i}g=\underset{g\in G}{\sum }c_{g}gX_{i}=\underset{g\in G}{\sum }%
c_{g}X_{i}^{g}g$. If $c_{g}\neq 0$, then $X_{i}^{g}=X_{i}$ . Since $X_{i}$
is arbitrary, $g=1$ and $v\in C[X].$

Now, $gv=v^{g}g=vg$ implies $v=v^{g}$ for all $g\in G$ and $v\in C[X]^{G}$.
\end{proof}

It is well known $[$Be$]$, $[$Stu$]$, $C[X]^{G}$ is an affine algebra and $%
C[X]^{G}\rightarrow C[X]$ an integral extension,, in particular, $\dim
C[X]^{G}=\dim C[X]$ and the maximal spectrums $\max specC[X]^{G}$ and $\max
specC[X]$ are isomorphic [Ku].

Moreover, it is known [Be] that given a prime $\emph{p}$ of $C[X]^{G}$ and
primes \emph{P}, \emph{Q }of $K[X]$ above \emph{p}, this is: \emph{P}$\cap
C[X]^{G}=\emph{Q}\cap C[X]^{G}=\emph{p}$, there exists $\sigma \in G$ such
that \emph{P}$^{\sigma }=\emph{Q}$ and for any $\sigma \in G$, \emph{P}$%
^{\sigma }\cap C[X]^{G}=\emph{P}\cap C[X]^{G}=\emph{p}$.

In particular for $\emph{m}$ a maximal ideal of $C[X]^{G}$ and maximal
ideals \emph{n}$_{1}$, \emph{n}$_{2}$ of $C[X]$ with \emph{n}$_{1}\cap
C[X]^{G}=$ \emph{n}$_{2}\cap $ $C[X]^{G}=\emph{m}$, there exists $\sigma \in
G$ with \emph{n}$_{2}$,=\emph{n}$_{1}^{\sigma }$ and for any $\sigma \in G$, 
\emph{n}$_{1}^{\sigma }$ is a maximal ideal with \emph{n}$_{1}^{\sigma }\cap
C[X]^{G}=$ \emph{n}$_{1}\cap $ $C[X]^{G}=\emph{m}$.

The points $v$ of $V=\underset{i=1}{\overset{n}{\oplus }}CX_{i}$ correspond
to maximal ideals \emph{n}$_{v}$ of $C[X]$ . In particular, $0\in V$
corresponds to the irrelevant maximal ideal $\emph{n}%
_{0}=(X_{1},X_{2},...X_{n})$ of $C[X]$.

By hypothesis, for $v\neq 0$, the orbit of $v$ under $G$, \emph{O}$(v)$= \{$%
\sigma (v)\mid \sigma \in G\}$ has order $\left\vert \emph{O}(v)\right\vert
=\left\vert G\right\vert .$The point $v$ corresponds to a maximal ideal 
\emph{n }such that \emph{n}$\cap $ $C[X]^{G}=\emph{m}$ and there are $%
\left\vert G\right\vert $ ideals above \emph{m }and they are precisely \{%
\emph{n}$^{\sigma }\mid \sigma \in G\}$. This is \emph{m }is an unramified
prime.

In the case $v=0$, for any $\sigma \in G$, $\sigma (v)=v$ and the maximal
ideal $\emph{n}_{0}$ satisfies $\emph{n}_{0}^{\sigma }=\emph{n}_{0}$ for all 
$\sigma \in G$.

The maximal ideal \emph{n}$_{0}\cap $ $C[X]^{G}=\emph{m}_{0}$ has a unique
maximal ideal, $\emph{n}_{0}$, above it.

We consider the two cases separately.

a) The ideal \emph{m }of $C[X]^{G}$ is maximal and different of $\emph{m}%
_{0} $.

Let \emph{n}$^{\sigma }$with $\sigma \in G$ be the set of all maximal ideals
of $C[X]$ above $\emph{m}_{0}$, hence $\emph{m}C[X]\subseteq \underset{%
\sigma \in G}{\cap }$\emph{n}$^{\sigma }$. It follows $\sqrt{\emph{m}C[X]}=%
\underset{\sigma \in G}{\cap }$\emph{n}$^{\sigma }$, the radical of \emph{m}$%
C[X]$.

It is clear that for any $\tau \in G$, $(\underset{\sigma \in G}{\cap }$%
\emph{n}$^{\sigma })^{\tau }=\underset{\sigma \in G}{\cap }$\emph{n}$%
^{\sigma \tau }=\underset{\sigma \in G}{\cap }$\emph{n}$^{\sigma }$ and $%
\sqrt{\emph{m}C[X]}$ is $G$-invariant.

By the Chinese Remainder Theorem, C$[$X$]$/$\sqrt{\emph{m}\text{C}[\text{X}]}
$= C$[$X$]$/$\underset{\sigma \in G}{\cap }$\emph{n}$^{\sigma }$=$\underset{%
\sigma \in G}{\tprod }$C$[$X$]$/\emph{n}$^{\sigma }$. Since $C[X]/\emph{n}%
\cong Cv$, $\underset{\sigma \in G}{\tprod }C[X]/$\emph{n}$^{\sigma }\cong 
\underset{\sigma \in G}{\tprod }Cv^{\sigma }$, with $1=\underset{\sigma \in G%
}{\sum }v^{\sigma }$, and $\{$ $v^{\sigma }\}_{\sigma \in G}$ is a complete
set of primitive orthogonal idempotents.

For the skew group algebra we have the following isomorphism:

$C[X]%
%TCIMACRO{\U{a8}}%
%BeginExpansion
\ddot{}%
%EndExpansion
\ast G$/$\sqrt{\emph{m}C[X]}\ast G\cong $($C[X]$/$\sqrt{\emph{m}C[X]}$)$\ast
G\cong $($\underset{\sigma \in G}{\tprod }Cv^{\sigma }$)$\ast G$.

Since $\underset{\sigma \in G}{\tprod }Cv^{\sigma }$semisimple, it follows
by $[]$, $C[X]%
%TCIMACRO{\U{a8}}%
%BeginExpansion
\ddot{}%
%EndExpansion
\ast G$/$\sqrt{\emph{m}C[X]}\ast G$ is semismple.

Let $S$ be $S=$ $\underset{\sigma \in G}{\tprod }Cv^{\sigma }$, the group $G$
acts transitively on the basis $\{$ $v^{\sigma }\}_{\sigma \in G}$ of $S.$

We claim $S\ast G$ is simple, it will be enough to prove that the center of $%
S\ast G$ is $C.$

It is clear that $C$ is contained in the center. Let $z\in Z(S\ast G)$ be an
element of the center, $z=\underset{\sigma \in G}{\sum }s_{\sigma }\sigma $.
Each $s_{\sigma }$ has form $s_{\sigma }=\underset{\tau \in G}{\sum }%
c_{\sigma ,v\tau }v^{\tau },$then $v_{\sigma }^{\rho }=\underset{\tau \in G}{%
\sum }c_{\sigma ,\tau }v^{\rho }v^{\tau }=c_{\sigma ,\rho }v^{\rho }.$

It follows $v^{\rho }z=\underset{\sigma \in G}{\sum }v^{\rho }s_{\sigma
}\sigma =\underset{\sigma \in G}{\sum }c_{\sigma ,\rho }v^{\rho }\sigma
=zv^{\rho }=\underset{\sigma \in G}{\sum }s_{\sigma }\sigma v^{\rho }=%
\underset{\sigma \in G}{\sum }s_{\sigma }v^{\rho \sigma }\sigma =\underset{%
\sigma \in G}{\sum }c_{\sigma ,\rho \sigma }v^{\rho \sigma }\sigma $.
Therefore: $c_{\sigma ,\rho \sigma }v^{\rho \sigma }=c_{\sigma ,\rho
}v^{\rho }$ for all $\rho $, $\sigma $and $c_{\sigma ,\rho \sigma
}=c_{\sigma ,\rho }=0$ for $\sigma \neq 1.$

It follows $z\in S$, this is $z=\underset{\sigma \in G}{\sum }c_{\sigma
}v^{\sigma }$ and $c_{\sigma }\in C.$Hence, $\tau z=\underset{\sigma \in G}{%
\sum }c_{\sigma }\tau v^{\sigma }=\underset{\sigma \in G}{\sum }c_{\sigma
}v^{\sigma \tau }\tau =\underset{\sigma \in G}{\sum }c_{\sigma }v^{\sigma
}\tau =z\tau $ and $\underset{\sigma \in G}{\sum }c_{\sigma }v^{\sigma \tau
}=\underset{\sigma \in G}{\sum }c_{\sigma }v^{\sigma }=\underset{\sigma \in G%
}{\sum }c_{\rho \tau ^{-1}}v^{\rho }$. It follows $c_{\rho \tau
^{-1}}=c_{\rho }$ for all $\rho $, $\tau $, in particular, $c_{1}=c_{\tau }$%
and $z=c_{1}$ $\underset{\tau \in G}{\sum }v^{\tau }$ = $c_{1}1$.

Since $S^{G}\subseteq Z(S\ast G)$ we have proved $S^{G}=Z(S\ast G)=C$.

The element $e=1/\left\vert G\right\vert \underset{g\in G}{\sum }g$ is an
idempotent of $KG$, hence a non zero idempotent of ($C[X]$/$\sqrt{\emph{m}%
C[X]}$)$\ast G$ and $e$(($C[X]$/$\sqrt{\emph{m}C[X]}$)$\ast G)\neq 0$.

Being the algebra ($C[X]$/$\sqrt{\emph{m}C[X]}$)$\ast G$ simple and the
ideal:

(($C[X]$/$\sqrt{\emph{m}C[X]}$)$\ast G)e($($C[X]$/$\sqrt{\emph{m}C[X]}$)$%
\ast G)$ non zero

($C[X]$/$\sqrt{\emph{m}C[X]}$)$\ast G=$(($C[X]$/$\sqrt{\emph{m}C[X]}$)$\ast
G)e($($C[X]$/$\sqrt{\emph{m}C[X]}$)$\ast G)$.

Case 2.

In this case we have : $C[X]%
%TCIMACRO{\U{a8}}%
%BeginExpansion
\ddot{}%
%EndExpansion
\ast G$/$\sqrt{\emph{m}_{0}C[X]}\ast G\cong $($C[X]$/$\emph{n}_{0}$)$\ast
G\cong CG$.

Let $L_{1}$, $L_{2}$,... $L_{s}$ the two sided maximal ideals of $CG$ and $%
\emph{L}_{1}$, $\emph{L}_{2}$,... $\emph{L}_{s}$ maximal two sided ideals $%
C[X]\ast G$ containing $\sqrt{\emph{m}_{0}C[X]}\ast G$. In particular, $%
C[X]\ast G/\emph{L}_{i}\cong CG/L_{i}\cong C$ and $\underset{i=1}{\overset{s}%
{\cap }}\emph{L}_{i}=\emph{n}_{0}\ast G$.

Let \{\emph{m}$_{i}\}_{i\in I}$ be the set of maximal ideals of $C[X]^{G}$
different from \emph{m}$_{0}$, we have a natural homomorphism:

$\Psi :$ $C[X]\ast G\rightarrow \underset{i\in I}{\tprod }C[X]$/$\sqrt{\emph{%
m}_{i}C[X]}$)$\ast G\times C[X]\ast G$/$\emph{n}_{0}\ast G$, given by $\Psi (%
\underset{g\in G}{\sum }r_{g}g)=((\underset{g\in G}{\sum }r_{g}g+\sqrt{\emph{%
m}_{i}C[X]}$)$\ast G)$, $\underset{g\in G}{\sum }r_{g}g+\emph{n}_{0}\ast G)$%
, which has kernel $\ker \Psi =\underset{_{i\in I}}{\cap }(\underset{\sigma
\in G}{\cap }$\emph{n}$_{i}^{\sigma })\ast G)\cap (\emph{n}_{0}\ast G)=(%
\underset{_{i\in I}}{\cap }(\underset{\sigma \in G}{\cap }$\emph{n}$%
_{i}^{\sigma })\cap \emph{n}_{0})\ast G=0$, since $(\underset{_{i\in I}}{%
\cap }(\underset{\sigma \in G}{\cap }$\emph{n}$_{i}^{\sigma })\cap \emph{n}%
_{0})$ is the intersection of all maximal ideals of $C[X]$. The map $\Psi $
is an injective ring homomorphism.

Denote by $\Sigma $ the product $\underset{i\in I}{\tprod }S_{i}\ast
G=\Sigma $ and $\overline{\Sigma }=\Sigma \times CG.$

The map $\Psi $sends the idempotent $e=1/\left\vert G\right\vert \underset{%
g\in G}{\sum }g$, into $\overset{\wedge }{e}=((\overline{e},e)$, where $%
\overline{e}=e+\sqrt{\emph{m}_{i}C[X]}$)$\ast G).$

Then $\overline{\Sigma }$ $\overset{\wedge }{e}$ $\overline{\Sigma }=(%
\underset{i\in I}{\tprod }S_{i}\ast GeS_{i}\ast G)\times CGeCG=\underset{%
i\in I}{\tprod }S_{i}\ast G\times Ce$=$\Sigma \times Ce$.

It is clear that $\Psi (C[X]\ast GeC[X]\ast G)$ is contained in $\overline{%
\Sigma }$ $\overset{\wedge }{e}$ $\overline{\Sigma }\cap \Psi (C[X]\ast
G)=(\Sigma \times Ce)\cap \Psi (C[X]\ast G)$.

We want to prove they are equal.

Let $((\overline{r}_{i}),$ $ce)$ be an element of $(\Sigma \times Ce)\cap
\Psi (C[X]\ast G)$. This means that there exists an element $r=\underset{%
g\in G}{\sum }r_{g}g\in C[X]\ast G$ such that $\overline{r}=\overline{r}_{i}$
for all $i\in I$ and $\overline{r}=c\overline{e}$ in $C[X]\ast Ge$/$\emph{n}%
_{0}\ast Ge\cong CGe.$Then $\underset{g\in G}{\sum }r_{g}g=ce+he$ with h a
polynomial without constant term. This implies for all $g\in G,$ $r_{g}=c+h$
and $\underset{g\in G}{\sum }r_{g}g=(c+h)e\in $ $C[X]\ast GeC[X]\ast G$. It
follows $(\Sigma \times Ce)\cap \Psi (C[X]\ast G)=\Psi (C[X]\ast GeC[X]\ast
G)$.

We have a commutative exact diagram:

\begin{center}
$%
\begin{array}{cccc}
& 0 &  & 0 \\ 
& \downarrow &  & \downarrow \\ 
0\rightarrow & \text{C[X]}\ast \text{GeC[X]}\ast \text{G} & \rightarrow & 
\underset{i\in I}{\tprod }S_{i}\ast G\times Ce \\ 
& \downarrow &  & \downarrow \\ 
0\rightarrow & C[X]\ast G & \rightarrow & \underset{i\in I}{\tprod }%
S_{i}\ast G\times CG \\ 
& \downarrow &  & \downarrow \\ 
0\rightarrow & \text{C[X]}\ast \text{G/C[X]}\ast \text{GeC[X]}\ast \text{G}
& \rightarrow & CG(1-e) \\ 
& \downarrow &  & \downarrow \\ 
& 0 &  & 0%
\end{array}%
$
\end{center}

It follows C[X]$\ast $G/C[X]$\ast $GeC[X]$\ast $G is a subalgebra of a
semisimple algebra, hence it is a finite dimensional $C$-algebra.

We can identify the category of C[X]$\ast $G/C[X]$\ast $GeC[X]$\ast $%
G-modules with the category of $C[X]\ast G$-modules $M$ with $eM=0.$

Let%
%TCIMACRO{\U{b4}}%
%BeginExpansion
\'{}%
%EndExpansion
s assume more generally that $R$ is a ring and $e$ an idempotent of $R$.
Then there is a functor $Hom_{R}(\func{Re},-):Mod_{R}\rightarrow Mod_{e\func{%
Re}}$, which is exact and it has a left adjoint $\func{Re}\otimes _{e\func{Re%
}}$- such that $Hom_{R}(\func{Re},\func{Re}\otimes _{e\func{Re}}M)=M.$

It follows $Hom_{R}(\func{Re},-)$ is a dense functor with kernel $Mod_{R/%
\func{Re}R}.$

The category \textsl{A }=$Mod_{R/\func{Re}R}=\{M\in Mod_{R}\mid eM=0\}$ is a
dense (Serre) subcategory of $Mod_{R}$. [P]

Given an exact sequence of $R$-modules: $0\rightarrow L\rightarrow
M\rightarrow N\rightarrow 0$ applying the exact functor $Hom_{R}(\func{Re}%
,-) $ we obtain an exact sequence of $e\func{Re}$ -modules: $0\rightarrow
eL\rightarrow eM\rightarrow eN\rightarrow 0$, hence $eM=0$ if and only if $%
eN=eL=0.$

Given a dense subcategory \textsl{A }of $Mod_{R}$ we define the
multiplicative system of all maps $f:M\rightarrow N$ such that the kernel
and the cokernel of $f$ are in \textsl{A}.

We define a quotient category $Mod_{R}/\QTR{sl}{A}=(Mod_{R})_{\Sigma }$ (see 
$[$Ga$]$, $[$P$]$, $[$Mi$]).$

The category $(Mod_{R})_{\Sigma }$ is abelian and the quotient functor $\pi
:Mod_{R}\rightarrow (Mod_{R})_{\Sigma }$ is exact an it has the following
universal property: Given an abelian category $\QTR{sl}{C}$ and an exact
functor $F:Mod_{R}\rightarrow \QTR{sl}{C}$ such that $F(X)=0$ for all $X\in 
\QTR{sl}{A}$, there is a unique exact functor $H:(Mod_{R})_{\Sigma
}\rightarrow \QTR{sl}{C}$ such that $H\pi =F$.

\begin{theorem}
Let $R$ be a ring and $e$ an idempotent of $R$ and \textsl{A }the category 
\textsl{A }=$\{M\in Mod_{R}\mid eM=0\}$. Then there is a commutative diagram
of categories and functors: $%
\begin{array}{ccc}
Mod_{R} & \overset{Hom_{R}(\func{Re},-)}{\rightarrow } & Mod_{e\func{Re}} \\ 
\downarrow \pi &  & \downarrow \cong \\ 
Mod_{R}/\QTR{sl}{A} & \overset{H}{\rightarrow } & Mod_{e\func{Re}}%
\end{array}%
$, with $H$ an equivalence.
\end{theorem}

\begin{proof}
1) The functor $H$ is dense.

Since $Hom_{R}(\func{Re},-)$ is dense, it follows $H$ is dense.

2) The functor $H$ is full.

Let $M$ and $N$ be $R$-modules and $\func{Re}\otimes _{e\func{Re}}eM\overset{%
\mu }{\rightarrow }M$ the map given by multiplication, taking the kernel and
the cokernel of $\mu $ we get an exact sequence:

$0\rightarrow L\rightarrow \func{Re}\otimes _{e\func{Re}}eM\overset{\mu }{%
\rightarrow }M\rightarrow M/\func{Re}M\rightarrow 0$ .

From the exact sequence:

$0\rightarrow eL\rightarrow e(\func{Re}\otimes _{e\func{Re}}eM)\overset{e\mu 
}{\rightarrow }eM\rightarrow e(M/\func{Re}M)\rightarrow 0$ and the fact $%
e\mu $ is an isomorphisms $eL=0=e(M/\func{Re}M).$

Hence the multiplication maps $\func{Re}\otimes _{e\func{Re}}eM\overset{\mu }%
{\rightarrow }M$ and $\func{Re}\otimes _{e\func{Re}}eN\overset{\mu ^{\prime }%
}{\rightarrow }N$ are isomorphisms in $Mod_{R}/\QTR{sl}{A}$.

Given a map $f:eM\rightarrow eN$ the "roof" $%
\begin{array}{ccccc}
&  & \func{Re}\otimes _{e\func{Re}}eM &  &  \\ 
\mu & \swarrow &  & \searrow & \mu ^{\prime }1\otimes f \\ 
M &  &  &  & N%
\end{array}%
$ is a map in $Mod_{R}/\QTR{sl}{A}$ such that when we apply the functor $%
Hom_{R}(\func{Re},-)$ we obtain maps: $%
\begin{array}{ccccc}
&  & Hom_{R}(\func{Re},\func{Re}\otimes _{e\func{Re}}eM) &  &  \\ 
Hom_{R}(\func{Re},\mu ) & \swarrow &  & \searrow & Hom_{R}(\func{Re},\mu
^{\prime }1\otimes f) \\ 
Hom_{R}(\func{Re},M) &  &  &  & Hom_{R}(\func{Re},N)%
\end{array}%
$, with $Hom_{R}(\func{Re},\mu )$, $Hom_{R}(\func{Re},\mu ^{\prime })$
isomorphisms.

Hence $H(\mu ^{\prime }1\otimes f\mu ^{-1})=f.$

3) The functor $H$ is faithful.

Let $%
\begin{array}{ccccc}
&  & W &  &  \\ 
s & \swarrow &  & \searrow & f \\ 
M &  &  &  & N%
\end{array}%
$ be a map in $Mod_{R}/\QTR{sl}{A.}$Then the kernel and the cokernel of $s$
are in \textsl{A}. Assume $H(fs^{-1})=0.$

Hence, $H(fs^{-1})=Hom_{R}(\func{Re}$, $f)Hom_{R}(\func{Re},s)^{-1}=0$ and $%
Hom_{R}(\func{Re},s)$ an isomorphism implies $Hom_{R}(\func{Re},f)=ef=0.$

From the commutative diagram: $%
\begin{array}{ccccc}
0\rightarrow & eW & \overset{j}{\rightarrow } & W &  \\ 
& ef\downarrow &  & f\downarrow &  \\ 
0\rightarrow & eN & \overset{j}{\rightarrow } & N & 
\end{array}%
$ we get $fj=0.$But $j:eW\rightarrow W$ is an isomorphism in $Mod_{R}/%
\QTR{sl}{A}$. Therefore: $f=0$ and $fs^{-1}=0$ in $Mod_{R}/\QTR{sl}{A}$.
\end{proof}

We come back now to the case of a finite subgroup $G$ of $Gl(n,C)$ such that
no $\sigma \in G$ different from the identity has a fixed point, $G$ acting
as a group of automorphism of the polynomial ring $C[X]$ in $n$ variables
the idempotent $e=1/\left\vert G\right\vert \underset{g\in G}{\sum }g$ of $%
CG $ is an idempotent of the skew group algebra $C[X]\ast G$ such that $%
eC[X]\ast Ge$ is isomorphic to the group of invariants $C[X]^{G}.$The
algebra $C[X]^{G}$ is a sub algebra of $C[X]$ and $C[X]$ is finitely
generated as $C[X]^{G}$-module. Let $f_{1}$, $f_{2}$,... $f_{m}$ be the
generators. The epimorphism $\rho :\underset{m}{\oplus }C[X]^{G}\rightarrow
C[X]$ given by $\rho (\gamma _{1}$,$\gamma _{2}$,...$\gamma _{m})=\overset{m}%
{\underset{i=1}{\sum }}\gamma _{i}f_{i}$ , extends to a map: $\overset{%
\wedge }{\rho }:$ $\underset{m}{\oplus }C[X]^{G}e\rightarrow C[X]e$ of left $%
C[X]^{G}$-modules $\overset{\wedge }{\rho }(\gamma _{1}e$,$\gamma _{2}e$,...$%
\gamma _{m}e)=\overset{m}{\underset{i=1}{\sum }}\gamma _{i}f_{i}e$ . Let $%
e\lambda =\lambda e=e\lambda e\in C[X]^{G}e$. Then $\overset{\wedge }{\rho }%
(\gamma _{1}e\lambda $,$\gamma _{2}e\lambda $,...$\gamma _{m}e\lambda )=%
\overset{\wedge }{\rho }(\gamma _{1}\lambda e$,$\gamma _{2}\lambda e$,...$%
\gamma _{m}\lambda e)=\overset{m}{\underset{i=1}{\sum }}\gamma _{i}\lambda
f_{i}e=(\overset{m}{\underset{i=1}{\sum }}\gamma _{i}f_{i}e)\lambda =(%
\overset{m}{\underset{i=1}{\sum }}\gamma _{i}f_{i}e)e\lambda $ and $\overset{%
\wedge }{\rho }$ is a map of $C[X]^{G}-eC[X]\ast Ge$ bimodules. If $M$ is a $%
eC[X]\ast Ge$-module of finite dimensional over $C$, from the epimorphism: $%
\overset{\wedge }{\rho }\otimes 1$ $:\underset{m}{\oplus }C[X]^{G}e\otimes
_{eC[X]\ast Ge}M\rightarrow C[X]e$ $\otimes _{eC[X]\ast Ge}M$ and the
isomorphism: $C[X]^{G}e\otimes _{eC[X]\ast Ge}M\cong eC[X]\ast Ge\otimes
_{eC[X]\ast Ge}M\cong M$ we obtain that $C[X]e$ $\otimes _{eC[X]\ast
Ge}M=C[X]\ast Ge$ $\otimes _{eC[X]\ast Ge}M$ is finite dimensional over $C$.

It is also clear that if $M$ is a $C[X]\ast G$-module finite dimensional,
then $eM=Hom_{C[X]\ast G}($ $C[X]\ast Ge,M)$ is finite dimensional over $C.$

To simplify the notation we will write $R=C[X]\ast G$ and $T=eC[X]\ast
Ge\cong C[X]^{G}$ and denote by \textsl{S}$_{R}$ and \textsl{S}$_{T}$ the
categories of finite dimensional $R$ and $T$-modules, respectively.

The categories \textsl{S}$_{R}$ and \textsl{S}$_{T}$ are dense subcategories
of the categories of finitely generated, $\func{mod}_{R}$, $\func{mod}_{T}$, 
$R$ and $T$ -modules, respectively.

Then we have:

\begin{theorem}
Let $G$ be a finite subgroup of $Gl(n,C)$ such that no $\sigma \in G$
different from the identity has a fixed point, $G$ acting as a group of
automorphism of the polynomial ring $C[X]$ in $n$ variables and $e$ the
idempotent $e=1/\left\vert G\right\vert \underset{g\in G}{\sum }g$ of the
skew group algebra $C[X]\ast G.$Writing $R=C[X]\ast G$ and $T=eC[X]\ast
Ge\cong C[X]^{G}$ and denoting by \textsl{S}$_{R}$ and \textsl{S}$_{T}$ the
categories of finite dimensional $R$ and $T$-modules, respectively and by $%
\func{mod}_{R}$, $\func{mod}_{T}$, the categories of finitely generated, $R$
and $T$ -modules, respectively. Then there is a commutative diagram of
categories and functors: $%
\begin{array}{ccc}
\func{mod}_{R} & \overset{Hom_{R}(\func{Re},-)}{\rightarrow } & \func{mod}%
_{T} \\ 
\downarrow \pi _{R} &  & \downarrow \pi _{T} \\ 
\func{mod}_{R}/\QTR{sl}{S}_{R} & \overset{H}{\rightarrow } & \func{mod}_{T}/%
\QTR{sl}{S}_{T}%
\end{array}%
$, with $H$ an equivalence.
\end{theorem}

Before proving the theorem, observe that there is a graded version.

The ring $e(C[X]\ast G)e$ is a positively graded $C$-algebra with grading $($
$e(C[X]\ast G)e)_{k}=e((C[X]\ast G)_{k})e$, hence the isomorphism of $C$%
-algebras $e(C[X]\ast G)e\cong C[X]^{G}$ induces a positive grading on $%
C[X]^{G}$. In what follows we will assume $C[X]^{G}$ has this grading and
denote by $R$ and $T$ the positively graded $C$-algebras $C[X]\ast G$ and $%
C[X]^{G}$, respectively. If we denote by $gr_{R}$ and $gr_{T}$ the
categories of finitely generated graded $R$ and $T$ modules, respectively,
and degree zero maps. Following $[],$ [], we call the quotient categories $%
gr_{R}/$\textsl{S}$_{R}$ and $gr_{T}/$ \textsl{S}$_{T}$ the categories of
tails $tails_{R}$ and $tails_{T}$, another notation used in $[$MV2$]$ is $%
Qgr_{R}$ and $Qgr_{T.}$

\begin{theorem}
Let $G$ be a finite subgroup of $Gl(n,C)$ such that no $\sigma \in G$
different from the identity has a fixed point, $G$ acting as a group of
automorphism of the polynomial ring $C[X]$ in $n$ variables and $e$ the
idempotent $e=1/\left\vert G\right\vert \underset{g\in G}{\sum }g$ of the
skew group algebra $C[X]\ast G.$Writing $R=C[X]\ast G$ and $T=eC[X]\ast
Ge\cong C[X]^{G}$ and considering both as positively graded algebras,
denoting by \textsl{S}$_{R}$ and \textsl{S}$_{T}$ the categories of finite
dimensional graded $R$ and $T$-modules, respectively and by $g_{R}$, $gr_{T}$%
, the categories of finitely generated, $R$ and $T$ -modules, respectively
and by $tails_{R}$ and $tails_{T}$ the quotient categories $gr_{R}/$\textsl{S%
}$_{R}$ and $gr_{T}/$ \textsl{S}$_{T}$. Then there is a commutative diagram
of categories and functors: $%
\begin{array}{ccc}
gr_{R} & \overset{Hom_{R}(\func{Re},-)}{\rightarrow } & gr_{T} \\ 
\downarrow \pi _{R} &  & \downarrow \pi _{T} \\ 
tails_{R} & \overset{H}{\rightarrow } & tails_{T}%
\end{array}%
$, with $H$ an equivalence.
\end{theorem}

We will prove only the ungraded case, the other follows with the same line
of arguments.

\begin{proof}
1) The functor $H$ is dense.

It follows as above from the fact $Hom_{R}(\func{Re},-)$ is dense.

2) The functor $H$ is full.

Let $X$, $Y$ be objects in $\func{mod}_{R}/\QTR{sl}{S}_{R}$ and consider a
map $\varphi :H(X)\rightarrow H(Y)$ in $\func{mod}_{T}/\QTR{sl}{S}_{T}$,
with $H(X)=Hom_{R}(\func{Re},X)$, $H(Y)=Hom_{R}(\func{Re},Y)$ and $\varphi $
a "roof" $%
\begin{array}{ccccc}
&  & W &  & f \\ 
s & \swarrow &  & \searrow &  \\ 
eX &  &  &  & eY%
\end{array}%
$, where the kernel $L$, and the cokernel $N$, of $s$ are of finite length.

Applying the tensor functor $\func{Re}\otimes _{e\func{Re}}-$ and composing
with multiplication we \newline
obtain the maps:

$%
\begin{array}{ccccccccc}
&  & \func{Re}\otimes _{e\func{Re}}s &  & \func{Re}\otimes _{e\func{Re}}W & 
&  & \func{Re}\otimes _{e\func{Re}}f &  \\ 
&  &  & \swarrow &  & \searrow &  &  &  \\ 
\mu &  & \func{Re}\otimes _{e\func{Re}}eX &  &  &  & \func{Re}\otimes _{e%
\func{Re}}eY &  & \mu ^{\prime } \\ 
& \swarrow &  &  &  &  &  & \searrow &  \\ 
X &  &  &  &  &  &  &  & Y%
\end{array}%
$

We also have exact sequences:

$\func{Re}\otimes _{e\func{Re}}L\overset{\func{Re}\otimes _{e\func{Re}}j}{%
\rightarrow }\func{Re}\otimes _{e\func{Re}}W\overset{\func{Re}\otimes _{e%
\func{Re}}s}{\rightarrow }\func{Re}\otimes _{e\func{Re}}\func{Im}%
s\rightarrow 0$

$\func{Re}\otimes _{e\func{Re}}$ $\func{Im}s\rightarrow \func{Re}\otimes _{e%
\func{Re}}eX\rightarrow \func{Re}\otimes _{e\func{Re}}N\rightarrow 0$

By the above observation, $\func{Re}\otimes _{e\func{Re}}L$ and $\func{Re}%
\otimes _{e\func{Re}}N$ are finite dimensional $C$-vector spaces, then $U=%
\func{Im}\func{Re}\otimes _{e\func{Re}}j=Ker\func{Re}\otimes _{e\func{Re}}s$
is finite dimensional.

We have a commutative exact diagram:

$%
\begin{array}{ccccccccc}
&  &  &  &  & \text{0} &  &  &  \\ 
&  &  &  &  & \downarrow &  &  &  \\ 
& \text{0} &  &  &  & \text{Z} &  &  &  \\ 
& \downarrow &  &  &  & \downarrow &  &  &  \\ 
\text{0}\rightarrow & \text{U} & \rightarrow & \text{Re}\otimes \text{W} & 
\rightarrow & \text{Re}\otimes \text{Ims} & \rightarrow & \text{0} &  \\ 
& \downarrow &  & \downarrow 1 &  & \downarrow &  & \downarrow &  \\ 
\text{0}\rightarrow & \text{V} & \rightarrow & \text{Re}\otimes \text{W} & 
\rightarrow & \text{Re}\otimes \text{eX} & \rightarrow & \text{Re}\otimes 
\text{N} & \rightarrow \text{0} \\ 
& \downarrow &  &  &  & \downarrow &  & \downarrow &  \\ 
& \text{Z} &  & \text{0} & \rightarrow & \text{Re}\otimes \text{N} &  & 
\text{Re}\otimes \text{N} & \rightarrow \text{0} \\ 
& \downarrow &  &  &  & \downarrow &  & \downarrow &  \\ 
& \text{0} &  &  &  & \text{0} &  & \text{0} & 
\end{array}%
$

Applying the functor $Hom_{R}(\func{Re},-)$ to the diagram we obtain the
commutative exact diagram:

$%
\begin{array}{ccccccccc}
&  &  &  &  & \text{0} &  &  &  \\ 
&  &  &  &  & \downarrow &  &  &  \\ 
& \text{0} &  &  &  & e\text{Z} &  &  &  \\ 
& \downarrow &  &  &  & \downarrow &  &  &  \\ 
\text{0}\rightarrow & e\text{U} & \rightarrow & \text{W} & \rightarrow & 
\text{Ims} & \rightarrow & 0 &  \\ 
& \downarrow &  & \downarrow 1 & \overset{s}{\searrow } & \downarrow i &  & 
\downarrow &  \\ 
\text{0}\rightarrow & e\text{V} & \rightarrow & \text{W} & \rightarrow & 
\text{eX} & \rightarrow & \text{N} & \rightarrow \text{0} \\ 
& \downarrow &  &  &  & \downarrow &  & \downarrow &  \\ 
& e\text{Z} &  & \text{0} & \rightarrow & \text{N} & \rightarrow & \text{N}
& \rightarrow \text{0} \\ 
& \downarrow &  &  &  & \downarrow &  & \downarrow &  \\ 
& \text{0} &  &  &  & \text{0} &  & \text{0} & 
\end{array}%
$

Since the map $i$ is a monomorphism $e$Z=0 and by Theorem ? $Z$ is finite
dimensional, but both $U$ and $Z$ finite dimensional implies $V$ is finite
dimensional.

By the above remark, $N$ of finite dimension implies Re$\otimes $N is of
finite dimension and both the kernel and cokernel of $\func{Re}\otimes _{e%
\func{Re}}s$ are in \textsl{S}$_{R}.$

It is clear that H($\mu $($\func{Re}\otimes $f)($\func{Re}\otimes $s)$^{-1}$%
=Hom$_{R}\func{Re}$,$\mu $($\func{Re}\otimes $f))Hom$_{R}$($\func{Re}$,$%
\func{Re}\otimes $s)$^{-1}$\newline
=fs$^{-1}.$

3) The functor $H$ is faithful.

Let $%
\begin{array}{ccccc}
&  & W &  & g \\ 
s & \swarrow &  & \searrow &  \\ 
X &  &  &  & Y%
\end{array}%
$ be a map in $\func{mod}_{R}/\QTR{sl}{S}_{R}$, where $s$ has kernel $L$ and
cokernel $N,$both finite dimensional. Assume $H(gs^{-1}=0.$

There exist an exact sequence: 0$\rightarrow e$L$\rightarrow e$W$\overset{es}%
{\rightarrow }e$X$\rightarrow e$N$\rightarrow $0 and $L$, $N\in \QTR{sl}{S}%
_{R}$ implies $eL$, $eN\in \QTR{sl}{S}_{T}.$Then we have in $\func{mod}_{T}/%
\QTR{sl}{S}_{T}$ the map:

$%
\begin{array}{ccccc}
&  & eW &  & eg \\ 
es & \swarrow &  & \searrow &  \\ 
eX &  &  &  & eY%
\end{array}%
$,

which by assumption satisfies ($eg)(es)^{-1}=H(gs^{-1})=0$.

Then we have a commutative diagram:

$%
\begin{array}{cccccc}
&  & e\text{W} &  &  & e\text{g} \\ 
e\text{s} & \swarrow & \uparrow & \text{r}_{1} & \searrow &  \\ 
e\text{X} &  & \text{Z} &  &  & e\text{Y} \\ 
& \nwarrow & \downarrow & \text{r}_{2} & \nearrow & \text{0} \\ 
t^{\prime } &  & \text{V} &  &  & 
\end{array}%
$

with $esr_{1}=t^{\prime }r_{2}$ having kernel and cokernel in $\QTR{sl}{S}%
_{T}$ and $egr_{1}=0.$

We want to see first that $r_{1}$ has kernel and cokernel in $\QTR{sl}{S}%
_{T} $.

We have a commutative exact diagram:

$%
\begin{array}{ccccccccc}
&  &  &  &  & \text{0} &  &  &  \\ 
&  &  &  &  & \downarrow &  &  &  \\ 
& \text{0} &  &  &  & e\text{L} &  &  &  \\ 
& \downarrow &  &  &  & \downarrow &  &  &  \\ 
\text{0}\rightarrow & \text{B }^{\prime } & \rightarrow & \text{Z} & \overset%
{r_{1}}{\rightarrow } & e\text{W} & \rightarrow & \text{D}^{\prime } & 
\rightarrow \text{0} \\ 
& \downarrow &  & 1\downarrow &  & e\text{s}\downarrow &  & u\downarrow & 
\\ 
\text{0}\rightarrow & \text{B} & \rightarrow & \text{Z} & \overset{esr_{1}}{%
\rightarrow } & e\text{X} & \rightarrow & \text{D} & \rightarrow \text{0} \\ 
&  &  &  &  & \downarrow &  &  &  \\ 
&  &  &  &  & e\text{N} &  &  &  \\ 
&  &  &  &  & \downarrow &  &  &  \\ 
&  &  &  &  & \text{0} &  &  & 
\end{array}%
$

with $B$, $D$, $eL$, $eN$ in $\QTR{sl}{S}_{T}$. We need to check that $%
B^{\prime }$ and $D^{\prime }$ are also in $\QTR{sl}{S}_{T}$.

From the commutative exact diagram:

$%
\begin{array}{ccccccc}
& \text{0} &  &  &  &  &  \\ 
& \downarrow &  &  &  &  &  \\ 
\text{0}\rightarrow & \text{B}^{\prime } & \rightarrow & \text{Z} & 
\rightarrow & \text{E}^{\prime } & \rightarrow \text{0} \\ 
& \downarrow &  & 1\downarrow &  & \downarrow &  \\ 
\text{0}\rightarrow & \text{B} & \rightarrow & \text{Z} & \rightarrow & 
\text{E} & \rightarrow \text{0} \\ 
&  &  &  &  & \downarrow &  \\ 
&  &  &  &  & \text{0} & 
\end{array}%
$

We obtain the following commutative exact diagram:

$%
\begin{array}{ccccccc}
& \text{0} &  & \text{0} &  & \text{0} &  \\ 
& \downarrow &  & \downarrow &  & \downarrow &  \\ 
\text{0}\longrightarrow & \text{E}^{\prime \prime } & \longrightarrow & e%
\text{L} & \longrightarrow & \text{D}^{\prime \prime } & \longrightarrow 
\text{0} \\ 
& \downarrow &  & \downarrow &  & \downarrow &  \\ 
\text{0}\rightarrow & \text{E}^{\prime } & \rightarrow & e\text{W} & 
\rightarrow & \text{D}^{\prime } & \rightarrow \text{0} \\ 
& \downarrow &  & e\text{s}\downarrow &  & u\downarrow &  \\ 
\text{0}\rightarrow & \text{E} & \rightarrow & e\text{X} & \rightarrow & 
\text{D} & \rightarrow \text{0} \\ 
& \downarrow &  & \downarrow &  & \downarrow &  \\ 
& \text{0} & \rightarrow & e\text{N} & \rightarrow & e\text{N} & \rightarrow 
\text{0} \\ 
&  &  & \downarrow &  & \downarrow &  \\ 
&  &  & \text{0} &  & \text{0} & 
\end{array}%
$

Since $D^{\prime \prime }$ is a quotient of $e$ $L$ and $\func{Im}u$ a
submodule of $D$ it follows both $D^{\prime \prime }$ and $\func{Im}u$ are
in $\QTR{sl}{S}_{T}.$Therefore: $D^{\prime }$ is in $\QTR{sl}{S}_{T}.$

From the equality $egr_{1}=0$ it follows $eg$ factors through the finite
dimensional module $D^{\prime }.$Tensoring with $\func{Re}$, the map: $\func{%
Re}\otimes eg:$ $\func{Re}\otimes _{e\func{Re}}eW\rightarrow \func{Re}%
\otimes _{e\func{Re}}eY$ factors through a finite dimensional $R$-module.

We have a commutative exact diagram:

$%
\begin{array}{ccc}
\func{Re}\otimes _{e\func{Re}}eW & \overset{\mu }{\rightarrow } & W \\ 
\func{Re}\otimes eg\downarrow &  & \downarrow g \\ 
\func{Re}\otimes _{e\func{Re}}eY & \overset{\mu ^{\prime }}{\rightarrow } & Y%
\end{array}%
$

where the multiplication maps $\mu $ and $\mu ^{\prime }$ have kernels and
cokernels in \textsl{S}$_{R}.$ It follows that a the nap $g\mu $ factors
through a module of finite length. This implies $g\mu =0$ in $\func{mod}_{R}/%
\QTR{sl}{S}_{R}$ and $\mu $ an isomorphism in $\func{mod}_{R}/\QTR{sl}{S}%
_{R} $ implies $g=0$ and $s^{-1}g=0.$
\end{proof}

For the benefit of the reader we prove the last claim in more detail.

\begin{lemma}
Assume $H:L\rightarrow M$ is a map in $\func{mod}_{R}$ which factors through
a finite dimensional module. Then $h=0$ in $\func{mod}_{R}/\QTR{sl}{S}_{R}$.
\end{lemma}

\begin{proof}
$h=vu$ with $u:L\rightarrow U$, $v:U\rightarrow M$ homomorphisms and $U$ a $%
R $-module of finite dimension. Changing $U$ by $\func{Im}u$, we can assume $%
u$ is an epimorphism.

We have a commutative exact diagram:

$%
\begin{array}{ccccccc}
&  &  &  &  & \text{0} &  \\ 
&  &  &  &  & \downarrow &  \\ 
& \text{0} &  &  &  & \text{U}^{\prime } &  \\ 
& \downarrow &  &  &  & \downarrow &  \\ 
\text{0}\rightarrow & \text{B}^{\prime } & \rightarrow & \text{L} & \overset{%
u}{\rightarrow } & \text{U} & \rightarrow \text{0} \\ 
& \downarrow &  & 1\downarrow &  & v\downarrow &  \\ 
\text{0}\rightarrow & \text{B} & \overset{j}{\rightarrow } & \text{L} & 
\overset{h}{\rightarrow } & \text{M} &  \\ 
& \downarrow &  &  &  &  &  \\ 
& \text{U}^{\prime } &  &  &  &  &  \\ 
& \downarrow &  &  &  &  &  \\ 
& \text{0} &  &  &  &  & 
\end{array}%
$

Since $u$ is an epimimorphism, then $\func{Im}v=\func{Im}h$ and $U^{\prime }$
are of finite dimension. Therefore the injective map $j:B\rightarrow L$ has
kernel and cokernel in \textsl{S}$_{R}$.

We have a commutative diagram:

$%
\begin{array}{cccccc}
&  & L &  &  & h \\ 
1 & \swarrow & \uparrow & \text{j} & \searrow &  \\ 
L &  & \text{B} &  &  & M \\ 
& \nwarrow & \downarrow & \text{j} & \nearrow & \text{0} \\ 
t^{\prime } &  & \text{L} &  &  & 
\end{array}%
$

We have proved $h=0$ in $\func{mod}_{R}/\QTR{sl}{S}_{R}$.
\end{proof}

We comeback to the last claim of the theorem

Let $s:Z\rightarrow L$ be a map in $\func{mod}_{R}$ whose kernel and
cokernel are finite dimensional, $h:L\rightarrow M$ a map such that $hs$
factors through a module of finite dimension and let $j:B\rightarrow Z$ be
the kernel of $hs$. Then $j$ has kernel and cokernel of finite dimension and 
$sj$ has kernel and cokernel of finite dimension. We have a commutative
diagram:

$%
\begin{array}{cccccc}
&  & \text{L} &  &  & h \\ 
1 & \swarrow & \uparrow & \text{sj} & \searrow &  \\ 
\text{L} &  & \text{B} &  &  & \text{M} \\ 
& \nwarrow & \downarrow & \text{j} & \nearrow & \text{0} \\ 
s &  & \text{Z} &  &  & 
\end{array}%
$

It follows $h=0$ in $\func{mod}_{R}/\QTR{sl}{S}_{R}$.

\begin{remark}
In our case $R=C[X]\ast G$, $T=e(C[X]\ast G)e\cong C[X]^{G}$, the algebra $%
C[X]^{G}$ is affine this is, the is a polynomial ring $C[Y_{1}$, $Y_{2}$,... 
$Y_{m}]=C[Y]$ and an ideal $I$ of $C[Y]$ such that $C[Y]/I\cong C[X]^{G}$ .
The ring $C[X]^{G}$ has the grading induced by the isomorphism $e(C[X]\ast
G)e\cong C[X]^{G}$, but in general $I$ is not an homogeneous ideal in the
natural grading of $C[Y]$.
\end{remark}

Changing notation, from the graded version of the theorem we have an
equivalence of categories: $Qgr_{C[X]\ast G}\cong Qgr_{C[X]^{G}}.$ This
equivalence induces at the level of bounded derived equivalences: $\emph{D}%
^{b}($ $Qgr_{C[X]\ast G})\cong \emph{D}^{b}(Qgr_{C[X]^{G}}).$

As remarked above, the exterior algebra in $n$ variables, $\Lambda _{n}$ is
the Yoneda algebra of $C[X]$ , $G$ acts as an automorphism group of $\Lambda
_{n}$ and the Yoneda algebra of $C[X]\ast G$ is $\Lambda _{n}$ $\ast G$.

By a theorem of $[MS$ and by $[MM]$ there is an equivalence of triangulated
categories: \underline{$gr$}$_{\Lambda _{n}\ast G}\cong \emph{D}^{b}($ $%
Qgr_{C[X]\ast G})$, where \underline{$gr$}$_{\Lambda _{n}\ast G}$ denotes
the stable category of finitely generated graded modules. Therefore: there
is an equivalence of triangulated categories: \underline{$gr$}$_{\Lambda
_{n}\ast G}\cong \emph{D}^{b}(Qgr_{C[X]^{G}})$.

We have proved the following:

\begin{corollary}
Let $G$ be a finite subgroup of $Gl(n,C)$ such that no $\sigma \in G$
different from the identity has a fixed point, $G$ acting as a group of
automorphism of the polynomial ring $C[X]$ in $n$ variables and $e$ the
idempotent $e=1/\left\vert G\right\vert \underset{g\in G}{\sum }g$ of the
skew group algebra $C[X]\ast G$ and let $\Lambda _{n}$ be the exterior
algebra in $n$ variables. Then there are isomorphisms of triangulated
categories: \underline{$gr$}$_{\Lambda _{n}\ast G}\cong \emph{D}%
^{b}(Qgr_{C[X]^{G}})\cong \emph{D}^{b}($ $Qgr_{C[X]\ast G})$. In particular,
the categories $\emph{D}^{b}(Qgr_{C[X]^{G}})$ and $\emph{D}^{b}($ $%
Qgr_{C[X]\ast G})$ have Auslander-Reiten triangles and they are of the form $%
ZA_{n}.$
\end{corollary}

\begin{proof}
For the proof we use the fact $\Lambda _{n}\ast G$ is selfinjective Koszul
and results from [MZ] .
\end{proof}

\subsection{Invariants for the Weyl algebra $A_{n}$ and the homogenized Weyl
algebra $B_{n}.$}

In this subsection we study the ring of invariants of the Weyl algebra $%
A_{n}^{G}$ with $G$ a finite subgroup of the automorphism group of $A_{n}.$%
We prove $A_{n}^{G}$ is a simple algebra Morita equivalent to the skew group
algebra $A_{n}\ast G$ . We then consider the homogenized Weyl algebra $B_{n}$
and a subgroup $G$ of the group of automorphisms of $B_{n}$ satisfying some
mild conditions, and prove for the invariant group $B_{n}^{G}$ and the skew
group algebra $B_{n}\ast G$ there is a a theorem relating the category of
finitely generated left $B_{n}^{G}$ modules and the category of finitely
generated $B_{n}\ast G$ modules, similar to the last theorem of the previous
section.

\begin{theorem}
Let $\Lambda $ be a simple algebra over a field $K$ and, $G$ a finite group
of automorphisms of $\Lambda .$Then the skew group algebra $\Lambda \ast G$
is simple.
\end{theorem}

\begin{proof}
Let $I$ be a non zero two sided ideal of $\Lambda \ast G$ and $%
r=a_{0}+a_{1}\sigma _{1}+...a_{t}\sigma _{t}$with $a_{i}\in \Lambda $ and $%
\sigma _{i}\in G$, if some $a_{i}\neq 0$ the element $r\sigma _{i}^{-1}\neq
0 $ is in $I$ and the coefficient of the identity is no zero, hence we can
assume all $a_{i}$ in the expression of $r$ are non zero and call $t+1=\ell
(r)$ the length of $r$. We can choose $r$ to be of minimal length among the
non zero elements of $I$. Consider the set $L_{0}$ defined by $%
L_{0}=\{b_{0}\mid b_{0}+b_{1}\sigma _{1}+...b_{t}\sigma _{t}\in I\}\cup
\{0\} $ . By definition $L_{0}$ is non zero, we prove it is a two sided
ideal of $\Lambda $.

Let $r_{1}=b_{0}+b_{1}\sigma _{1}+...b_{t}\sigma _{t}$ and $%
r_{2}=c_{0}+c_{1}\sigma _{1}+...c_{t}\sigma _{t}$ be two elements of $I.$%
Then $r_{1}+r_{2}=(b_{0}+c_{0})+(b_{1}+c_{1})\sigma
_{1}+...(b_{t}+c_{t})\sigma _{t}$ is in $I$ and $\ell (r_{1}+r_{3})\leq \ell
(r_{1})$, by minimality either $r_{1}+r_{2}=0$ or $\ell (r_{1}+r_{3})=\ell
(r_{1})$ in any case $b_{0}+c_{0}\in L_{0}.$

Let $a$ be a non zero element of $\Lambda $. Then $ar_{1}=ab_{0}+ab_{1}%
\sigma _{1}+...ab_{t}\sigma _{t}$ and $r_{1}a=b_{0}a+b_{1}a^{\sigma
_{1}}\sigma _{1}+...b_{t}a^{\sigma t}\sigma _{t}$ are in $I$ with $ab_{0}$
and $b_{0}a$ in $L_{0}.$Hence $L_{0}$ is a two sided ideal.

It follows $L_{0}=\Lambda $ and there is an expression of minimal length $%
1+b_{1}\sigma _{1}+...b_{t}\sigma _{t}$ in $I$.

If $a_{0}+a_{1}\sigma _{1}+...a_{t}\sigma _{t}$ is another non zero
expression in $I$, then $(a_{0}b_{1-}a_{1})\sigma
_{1}+..(a_{0}b_{t}.-a_{t})\sigma _{t}$ $\in I$ and by minimality $%
a_{0}b_{i}-a_{i}=0$ for $i\neq 0$. This is: $a_{0}+a_{1}\sigma
_{1}+...a_{t}\sigma _{t}=a_{0}(1+b_{1}\sigma _{1}+...b_{t}\sigma _{t}).$

Multiplying by $\sigma _{i}^{-1}$ we obtain the expression $%
b_{i}+b_{1}\sigma \sigma _{i}^{-1}+...b_{i-1}\sigma _{i-1}+\sigma
_{i}^{-1}+b_{i+1}\sigma _{i+1}\sigma _{i}^{-1}+...b_{t}\sigma _{t}\sigma
_{i}^{-1}$ belongs to $I$ is non zero and of minimal length. We define as
above the set $L_{i}=\{$ $a_{0}\mid a_{0}+a_{1}\sigma \sigma
_{i}^{-1}+...a_{i-1}\sigma _{i-1}+a_{i}\sigma _{i}^{-1}+a_{i+1}\sigma
_{i+1}\sigma _{i}^{-1}+...a_{t}\sigma _{t}\sigma _{i}^{-1}\in I\}$. As
before $L_{i}$ is a non zero two sided ideal of $A_{n}.$Then there is an
expression $1+c_{1}\sigma \sigma _{i}^{-1}+...c_{i-1}\sigma
_{i-1}+c_{i}\sigma _{i}^{-1}+c_{i+1}\sigma _{i+1}\sigma
_{i}^{-1}+...c_{t}\sigma _{t}\sigma _{i}^{-1}$ in $I$ and $b_{i}+b_{1}\sigma
\sigma _{i}^{-1}+...b_{i-1}\sigma _{i-1}+\sigma _{i}^{-1}+b_{i+1}\sigma
_{i+1}\sigma _{i}^{-1}+...b_{t}\sigma _{t}\sigma
_{i}^{-1}=b_{i}(1+c_{1}\sigma \sigma _{i}^{-1}+...c_{i-1}\sigma
_{i-1}+c_{i}\sigma _{i}^{-1}+c_{i+1}\sigma _{i+1}\sigma
_{i}^{-1}+...c_{t}\sigma _{t}\sigma _{i}^{-1}).$In particular, $b_{i}c_{i}=1$
and $b_{i}\in C-\{0\}$ is a unity. Since the element was arbitrary all
coefficients in $1+b_{1}\sigma _{1}+...b_{t}\sigma _{t}$ are units, this is
they are non zero complex numbers.

Assume the length $t>0.$

Let \{ $X_{i}$\}$_{i\in \Phi }$ be a set of algebra generators of $\Lambda $
as $K$-algebra. Then $X_{i}+X_{i}b_{1}\sigma _{1}+...X_{i}b_{t}\sigma _{t}$
and $X_{i}$ $+b_{1}X_{i}^{\sigma _{1}}\sigma _{1}+...b_{t}X_{i}^{\sigma
_{t}}\sigma _{t}$ are elements of $I$ and $b_{1}(X_{i}-X_{i}^{\sigma
_{1}})\sigma _{1}+...b_{t}(X_{i}-X_{i}^{\sigma _{t}})\sigma _{t}$ is in $I.$%
By minimality $X_{i}=X_{i}^{\sigma _{1}}$ for an arbitrary $X_{i}$. But if $%
\sigma _{1}$ fixes all the generators of $\Lambda $. then $\sigma _{1}=1$, a
contradiction. It follows $t=0$ and $\Lambda =\{a\in \Lambda \mid a1\in I\}.$%
Therefore $1\in I$ and $I$ $=\Lambda \ast G.$
\end{proof}

As a corollary we obtain the following:

\begin{corollary}
Let A$_{n}$=C\TEXTsymbol{<}X$_{1}$,X$_{2}$,...X$_{n}$,Y$_{1}$,Y$_{2}$,...Y$%
_{n}$\TEXTsymbol{>}/\{[X$_{i}$,X$_{j}$],[Y$_{i}$,Y$_{j}$],[X$_{i}$,Y$_{j}$]-$%
\delta _{ij}$\} be a Weyl algebra and, $G$ a finite group of automorphisms
of $A_{n}.$Then the skew group algebra $\Lambda \ast G$ is simple.
\end{corollary}

Another consequence of the theorem is the following:

\begin{theorem}
Let $\Lambda $ be a simple algebra over a field $K$ and, $G$ a finite group
of automorphisms of $\Lambda .$Then the algebra of invariants $\Lambda ^{G}$
is simple.
\end{theorem}

\begin{proof}
Let $I$ be a non zero ideal of $\Lambda ^{G}.$Then $Ie$ $=eIe$ is a non zero
ideal of $\Lambda ^{G}e=e\Lambda \ast Ge$.

The ideal $\Lambda \ast GeIe\Lambda \ast G$ of $\Lambda \ast G$ is non zero
otherwise, e$\Lambda \ast GeIe\Lambda \ast Ge=eIe=0.$Therefore $\Lambda \ast
GeIe\Lambda \ast G$ = $\Lambda \ast G$ and $eIe=e\Lambda \ast GeIe\Lambda
\ast G$ $e=e\Lambda \ast Ge$.
\end{proof}

\begin{corollary}
Let A$_{n}$=C\TEXTsymbol{<}X$_{1}$,X$_{2}$,...X$_{n}$,Y$_{1}$,Y$_{2}$,...Y$%
_{n}$\TEXTsymbol{>}/\{[X$_{i}$,X$_{j}$],[Y$_{i}$,Y$_{j}$],[X$_{i}$,Y$_{j}$]-$%
\delta _{ij}$\} be a Weyl algebra and, $G$ a finite group of automorphisms
of $A_{n}.$Then the algebra of invariants $A_{n}^{G}$ is simple.
\end{corollary}

We can prove now the following:

\begin{theorem}
Let $\Lambda $ be a simple algebra over a field $K$ and, $G$ a finite group
of automorphisms of $\Lambda $. Then the algebra of invariants $\Lambda ^{G}$
and the skew group algebra $\Lambda \ast G$ are Morita equivalent.
\end{theorem}

\begin{proof}
Consider the idempotent of $\Lambda \ast G$, $e=1/\left\vert G\right\vert 
\underset{\sigma \in G}{\sum }\sigma $. We prove first that the functor $%
F=Hom_{\Lambda \ast G}(\Lambda \ast Ge,-):Mod_{\Lambda \ast G}\rightarrow
Mod_{e\Lambda \ast Ge}$ has zero kernel.

A $\Lambda \ast G$-module $M$ is in the kernel of $F$ if and only if $eM=0,$%
this is: if and only if $\Lambda \ast Ge\Lambda \ast GM=0$ . But $\Lambda
\ast G$ simple implies $\Lambda \ast G=\Lambda \ast Ge\Lambda \ast G$ and $M$
is in the kernel of $F$ if and only if $M=0.$

The functor $F$ is always dense.

Let $f:M\rightarrow N$ be a map of $\Lambda \ast G$-modules and $%
ef:eM\rightarrow eN$ the restriction.

There is a commutative diagram:

$%
\begin{array}{ccc}
\Lambda \ast Ge\otimes \otimes _{e\Lambda \ast Ge}eM & \overset{\mu }{%
\rightarrow } & M \\ 
\downarrow \Lambda \ast Ge\otimes ef &  & f\downarrow \\ 
\Lambda \ast Ge\otimes _{e\Lambda \ast Ge}eN & \overset{\mu ^{\prime }}{%
\rightarrow } & N%
\end{array}%
$

with $\mu $ and $\mu ^{\prime }$ multiplication. The kernels of $\mu $ and $%
\mu ^{\prime }$ are annihilated by $e,$ hence they are zero and the cokernel
of $\mu $ is $M/\Lambda \ast Ge\Lambda \ast GM=0.$

Similarly for $\mu ^{\prime }.$It follows both $\mu $ and $\mu ^{\prime }$
are isomorphisms.

Therefore $ef=0$ implies $f=0$ and $F$ is faithful.

Let $g:M^{\prime }\rightarrow N^{\prime }$be a map of $e\Lambda \ast Ge$%
-modules. Since $F$ is dense, there exists $\Lambda \ast G$-modules $M$ and $%
N$ such that $eM\cong M^{\prime }$and $eN\cong N^{\prime }$and $g$ can be
identified with a map $g:eM\rightarrow eN.$There are maps:

$%
\begin{array}{ccc}
\Lambda \ast Ge\otimes \otimes _{e\Lambda \ast Ge}eM & \overset{\mu }{%
\rightarrow } & M \\ 
\downarrow \Lambda \ast Ge\otimes g &  &  \\ 
\Lambda \ast Ge\otimes _{e\Lambda \ast Ge}eN & \overset{\mu ^{\prime }}{%
\rightarrow } & N%
\end{array}%
$

with $\mu $, $\mu ^{\prime }$isomorphisms and the map $f=\mu ^{\prime
}\Lambda \ast Ge\otimes g\mu ^{-1}$is such that

$Hom_{\Lambda \ast G}(\Lambda \ast Ge,f)=ef=g.$

Therefore $Hom_{\Lambda \ast G}(\Lambda \ast Ge,-)$ is full.
\end{proof}

\begin{corollary}
Let A$_{n}$=C\TEXTsymbol{<}X$_{1}$,X$_{2}$,...X$_{n}$,Y$_{1}$,Y$_{2}$,...Y$%
_{n}$\TEXTsymbol{>}/\{[X$_{i}$,X$_{j}$],[Y$_{i}$,Y$_{j}$],[X$_{i}$,Y$_{j}$]-$%
\delta _{ij}$\} be a Weyl algebra and, $G$ a finite group of automorphisms
of $A_{n}.$Then the algebra of invariants $A_{n}^{G}$ and the skew group
algebra $A_{n}\ast G$ are Morita equivalent.
\end{corollary}

\begin{corollary}
Let $G$ be a finite subgroup of $Sl(2,C)$ acting as automorphism group of $%
A_{1}$. Then $A_{1}\ast G$ and $A_{1}^{G}$ are simple and Morita equivalent
to the deformed preprojective algebra. In particular the deformed
preprojective algebra is simple.
\end{corollary}

For more results on the deformed preprojective algebra we refer to [C-BH].

We analyze next the relations between $\func{mod}_{B_{n}\ast G}$ and $\func{%
mod}_{B_{n}^{G}}$ for a special class of subgroups of automorphisms of $%
B_{n} $, which include finite products of finite subgroups of $Sl(2,C$).

We will need the following:

\begin{theorem}
Let $\Lambda $ be a noetherian algebra over a field $K$ and, $G$ a finite
group of automorphisms of $\Lambda $ such that $\Lambda $ is finitely
generated as left module over the ring of invariants $\Lambda ^{G}$. Then $%
\Lambda ^{G}$ is a noetherian algebra.
\end{theorem}

\begin{proof}
Let $I$ be a left ideal of $\Lambda ^{G}$. Then $\Lambda I$ is a left ideal
and by hypothesis, $\Lambda I$ is finitely generated as $\Lambda $-module
and $\Lambda $ finitely generated over $\Lambda ^{G}$ implies $\Lambda I$ is
a finitely generated $\Lambda ^{G}$-module. Let \{$\overset{n_{j}}{\underset{%
i=1}{\sum }}b_{i}^{j}u_{i}^{j}$, $1\leq j\leq m\mid b_{i}^{j}\in \Lambda $, $%
u_{i}^{j}\in I$\} be a set of generators of $\Lambda I$. Then \{$%
b_{i}^{j}u_{i}^{j}$, $1\leq j\leq m\mid b_{i}^{j}\in \Lambda $, $%
u_{i}^{j}\in I$\} is also a set of generators of $\Lambda I,$after re
indexing we have a set of generators $\{b_{i}u_{i}$, $1\leq i\leq n\mid
b_{i}\in \Lambda $, $u_{i}\in I\}$ . Let $x$ be an element of $I$. Then for $%
1\leq i\leq n$, there exists $c_{i}\in \Lambda ^{G}$ such that $x=\underset{%
i=1}{\overset{m}{\sum }}c_{i}b_{i}u_{i}$. Then for $\sigma \in G$ , $%
x=x^{\sigma }=\underset{i=1}{\overset{m}{\sum }}c_{i}^{\sigma }b_{i}^{\sigma
}u_{i}^{\sigma }=\underset{i=1}{\overset{m}{\sum }}c_{i}b_{i}^{\sigma }u_{i}$
and $x=1/\left\vert G\right\vert \underset{\sigma \in G}{\sum }\underset{i=1}%
{\overset{m}{\sum }}c_{i}b_{i}^{\sigma }u_{i}=1/\left\vert G\right\vert 
\underset{i=1}{\overset{m}{\sum }}\underset{\sigma \in G}{\sum }%
c_{i}b_{i}^{\sigma }u_{i}=\underset{i=1}{\overset{m}{\sum }}\underset{}{%
c_{i}(1/\left\vert G\right\vert \text{ }\underset{\sigma \in G}{\sum }}%
b_{i}^{\sigma })u_{i}=\underset{i=1}{\overset{m}{\sum }}c_{i}t_{i}u_{i}$,
with $c_{i}$, $t_{i}\in \Lambda ^{G}$.

Therefore: $I$ is finitely generated as $\Lambda ^{G}$-module.
\end{proof}

\begin{proposition}
Let $G$ be a finite group of grade preserving automorphisms of $B_{n}$
fixing $Z.$Then $B_{n}$ is a finitely generated left (right) $B_{n}^{G}$%
-module.
\end{proposition}

\begin{proof}
Since $G$ is a group of grade preserving automorphisms $B_{n}^{G}$ is a
positively graded ring with ($B_{n}^{G})_{i}=\{b\in (B_{n})_{i}\mid \sigma
(b)=b$ for all $\sigma \in G\}$ and the inclusion $j_{1}:B_{n}^{G}%
\rightarrow B_{n}$ is a homomorphism of graded $C$-algebras.

Let $C_{n}$ be the polynomial ring in $2n$ variables, $G$ acts as an
automorphism group of $C_{n}$, let $C_{n}^{G}$ be the ring of invariants and 
$j_{0}:C_{n}^{G}\rightarrow C_{n}$ the inclusion as a graded subring.

We have a commutative exact diagram:

$%
\begin{array}{ccccccc}
0\rightarrow & ZB_{n}^{G} & \rightarrow & B_{n}^{G} & \rightarrow & C_{n}^{G}
& \rightarrow 0 \\ 
& \downarrow j_{2} &  & \downarrow j_{1} &  & \downarrow j_{0} &  \\ 
0\rightarrow & ZB_{n} & \rightarrow & B_{n} & \rightarrow & C_{n} & 
\rightarrow 0%
\end{array}%
$

We know $C_{n}$ is a finitely generated $C_{n}^{G}$-module. $[]$ and we can
choose homogeneous generators $c_{1}$, $c_{2}$,... $c_{t}$ of $C_{n}$ as $%
C_{n}^{G}$-module . Let $b_{1}$, $b_{2}$,... $b_{t}$ be homogeneous elements
of $B_{n}$ such that $b_{i}+ZB_{n}=c_{i}$ for $1\leq i\leq t$.

Let $b$ be an element of $B_{n}$ of degree $d.$Then $b+ZB_{n}=\underset{i=1}{%
\overset{t}{\sum }}r_{i}^{0}b_{i}+ZB_{n}$, with $r_{i}^{0}\in B_{n}^{G}$ and 
$b=\underset{i=1}{\overset{t}{\sum }}r_{i}^{0}b_{i}+Z\mu _{1}$ .

Since $b,b_{1}$, $b_{2}$,... $b_{t}$ are homogeneous we can choose $%
r_{i}^{0} $ and $\mu _{1}$ homogeneous elements with degree$r_{i}^{0}+$%
degree $b_{i}$= $d$ and degree$\mu _{1}=d-1.$

Then $\mu _{1}=\underset{i=1}{\overset{t}{\sum }}r_{i}^{1}b_{i}+Z\mu _{2}$
with $\mu _{2}$ homogeneous of degree $d-2.$Continuing by induction we
obtain $\mu _{i}$ homogeneous of degree $d-i$, in particular $\mu _{d}$ has
degree zero, which mans it is a constant and we get:

$b=\underset{i=1}{\overset{t}{\sum }}r_{i}^{0}b_{i}+\underset{j=1}{\overset{%
d-1}{\sum }}\underset{i=1}{\overset{t}{\sum }}$ $Z^{j}r_{i}^{j}b_{i}+Z^{d}k$%
, $k$ a complex number. Then $b=\underset{i=1}{\overset{t}{\sum }}(\underset{%
j=0}{\overset{d-1}{\sum }}Z^{j}r_{i}^{j})b_{i}+Z^{d}k$, with $%
Z^{j}r_{i}^{j}\in B_{n}^{G}$ and $1,$ $b_{1}$, $b_{2}$,... $b_{t}$ generate $%
B_{n}$ as left $B_{n}^{G}$-module.
\end{proof}

\begin{corollary}
Let $G$ be a finite group of grade preserving automorphisms of $B_{n}$
fixing $Z.$Then $B_{n}^{G}$ is a noetherian algebra.
\end{corollary}

\begin{lemma}
Let $B_{n}$ be the homogenized Weyl $K$-algebra over an infinite field $H$.
Then $\underset{c\in K}{\cap }(Z-c)B_{n}=0$.
\end{lemma}

\begin{proof}
Since $B_{n}$ has a Poincare-Birkoff basis with $Z$ in the center, any
element of $B_{n}$ is a polynomial $g/X,Y,Z)$ in $X%
%TCIMACRO{\U{b4}}%
%BeginExpansion
{\acute{}}%
%EndExpansion
s,$ $Y%
%TCIMACRO{\U{b4}}%
%BeginExpansion
{\acute{}}%
%EndExpansion
s$ and $Z$. Given $d\in K$ $Z-d$ divides $g$ if and only if $g(X,Y,d)=0.$

It is clear that $g(X,Y,Z)=q(X,Y,Z)(Z-d)$ implies $g(X,Y,d)=0.$ We can write 
$g$ as a polynomial in $Z,$ $%
g(X,Y,Z)=g_{0}(X,Y)+g_{1}(X,Y)Z+g_{2}(X,Y)Z^{2}+...g_{t}(X,Y)Z^{t}$ with $%
g_{i}(X,Y)$ polynomials in $X%
%TCIMACRO{\U{b4}}%
%BeginExpansion
{\acute{}}%
%EndExpansion
s$ and $Y%
%TCIMACRO{\U{b4}}%
%BeginExpansion
{\acute{}}%
%EndExpansion
s$.

Then $g_{i}(X,Y)Z^{i}=g_{i}(X,Y)((Z-d)+d)^{i}=g_{i}^{\prime
}(X,Y,Z)(Z-d)+g_{i}(X,Y)d^{i}$ and $g(X,Y,Z)=q(X,Y,Z)(Z-d)+g(X,Y,d).$%
Therefore $g(X,Y,d)=0$ implies $Z-d$ divides $g$.

Assume now $0\neq h\in \underset{c\in K}{\cap }(Z-c)B_{n}.$ Then $%
h=(Z-c_{1})q_{1}=(Z-c_{2})f$ with $c_{1}\neq c_{2}.$Then $%
(c_{2}-c_{1})q_{1}(X,Y,c_{2})=(c_{2}-c_{2})f=0$ implies $Z-c_{2}$ divides $%
q_{1}(X,Y,Z)$ and $h=(Z-c_{1})(Z-c_{2})q_{2}.$

Assume $h$ is a polynomial in $Z$ of degree $m.$Continuing by induction we
obtain $h=(Z-c_{1})(Z-c_{2})...(Z-c_{m+1})q_{m+1}$ a contradiction.

Therefore: $\underset{c\in K}{\cap }(Z-c)B_{n}=0$.
\end{proof}

In what remains of the section we want to extend the theorems of the
previous subsection to the homogenized Weyl algebras, the following result
will be crucial:

\begin{proposition}
Let $B_{n}$ =C\TEXTsymbol{<}X$_{1}$,X$_{2}$,...X$_{n}$,Y$_{1}$,Y$_{2}$,...Y$%
_{n}$,Z\TEXTsymbol{>}/\{[X$_{i}$,X$_{j}$],[Y$_{i}$,Y$_{j}$],[X$_{i}$,Y$_{j}$%
]-$\delta _{ij}$Z$^{2}$,[X$_{i}$,Z],[Y$_{i}$,Z]\} be the homogenized Weyl
algebra in 2n+1 variables, $G$ a finite group of grade preserving
automorphisms of $B_{n}$ and $e$ the idempotent of $B_{n}\ast G$, $%
e=1/\left\vert G\right\vert \underset{\sigma \in G}{\sum }\sigma $.
Moreover, assume $B_{n}$ satisfies the following conditions:

i) For all $\sigma \in G$, $\sigma (Z)=Z.$

ii) The $C$-subspace $V=\overset{n}{\underset{i=1}{\oplus }}CX_{i}\oplus 
\overset{n}{\underset{j=1}{\oplus }}CY_{j}$ of $B_{n}$ is $G$- invariant.

iii) If $v\in V$, $v\neq 0$ and $\sigma (v)=v$, then $\sigma =1.$

Then $B_{n}\ast G/B_{n}\ast GeB_{n}\ast G$ is a finite dimensional $C$%
-algebra.
\end{proposition}

\begin{proof}
We will use the commutative exact diagram:

$%
\begin{array}{ccccccc}
0\rightarrow & ZB_{n}^{G} & \rightarrow & B_{n}^{G} & \rightarrow & C_{n}^{G}
& \rightarrow 0 \\ 
& \downarrow j_{2} &  & \downarrow j_{1} &  & \downarrow j_{0} &  \\ 
0\rightarrow & ZB_{n} & \rightarrow & B_{n} & \rightarrow & C_{n} & 
\rightarrow 0%
\end{array}%
$

Let \{$\overline{m}_{i}\}_{i\in I}$ be the set of maximal ideals of $%
C_{n}^{G}$ , for each $\overline{m}_{i}$ an ideal $\overline{n}_{i}$ of $%
C_{n}$ such that $C_{n}\cap \overline{n}_{i}=\overline{m}_{i}.$ Let $%
\overline{n}_{0}$= $(X_{1}$, $X_{2}$,... $X_{n}$, $Y_{1}$, $Y_{2}$,... $%
Y_{n})$ be the maximal irrelevant ideal of $C_{n}$ and $\overline{m}%
_{0}=C_{n}\cap \overline{n}_{0}$.

We saw in Proposition ?, that given any ideal $\overline{m}_{i}\neq 
\overline{m}_{0}$ , \{$\overline{n}_{i}^{\sigma }$\}$_{\sigma \in G}$ , such
that $\overline{n}_{i}^{\sigma }=\overline{n}_{i}^{\tau }$ implies $\sigma
=\tau $, is the set of maximal ideals of $C_{n}$ above $\overline{m}_{i}$.
In particular, the radical $\sqrt{\overline{m}_{i}C_{n}}$ =$\underset{\sigma
\in G}{\cap }\overline{n}_{i}^{\sigma }$ is a $G$-invariant ideal of $C_{n}$%
, and $\sqrt{\overline{m}_{0}C_{n}}=\overline{n}_{0}.$

By the isomorphism theorems, there exists maximal ideals \{ $m_{i}$\}$_{i\in
I}$ of $B_{n}^{G}$ containing $ZB_{n}^{G}$, such that $m_{i}/ZB_{n}^{G}\cong 
\overline{m}_{i}$ and maximal ideals \{$n_{i}$\}$_{i\in I}$ of $B_{n}$
containing $ZB_{n}$ such that $n_{i}/ZB_{n}\cong \overline{n}_{i}.$

From the isomorphisms: $n_{i}/ZB_{n}\cap B_{n}^{G}/ZB_{n}^{G}=n_{i}\cap
B_{n}^{G}/ZB_{n}^{G}=m_{i}/ZB_{n}^{G}$, we get $n_{i}\cap B_{n}^{G}=m_{i}$
and $n_{i}^{\sigma }\cap B_{n}^{G}=m_{i}$ for all $\sigma \in G.$ It follows 
$L_{i}=\underset{\sigma \in G}{\cap }n_{i}^{\sigma }$ is a $G$-invariant
ideal of $B_{n}$, with $L_{0}=n_{0}$.

We have $ZB_{n}\subseteq \underset{i\in I}{\cap }\underset{\sigma \in G}{%
\cap }n_{i}^{\sigma }=\underset{i\in I}{\cap }L_{i}$ and $\underset{i\in I}{%
\cap }L_{i}/ZB_{n}=\underset{i\in I}{\cap }\underset{\sigma \in G}{\cap }%
\overline{n}_{i}^{\sigma }=0.$Therefore: $\underset{i\in I}{\cap }\underset{%
\sigma \in G}{\cap }n_{i}^{\sigma }=ZB_{n}$.

It was proved in [], that for any $c\in C-\{0\}$ there is an isomorphism $%
B_{n}/(Z-c)B_{n}\cong A_{n},$ hence each $(Z-c)B_{n}$ is a maximal ideal of $%
B_{n}.$

Then we have: $\underset{i\in I}{\cap }\underset{\sigma \in G}{\cap }%
n_{i}^{\sigma }\cap \underset{c\in C-\{0\}}{\cap }(Z-c)B=\underset{c\in C}{%
\cap }(Z-c)B=0$.

We have a ring homomorphism:

$\psi :B_{n}\ast G\rightarrow \underset{i\in I-\{0\}}{\dprod }B_{n}\ast
G/L_{i}\ast G\times B_{n}\ast G/n_{0}\ast G\times \underset{c\in C-\{0\}}{%
\dprod }B_{n}\ast G/(Z-c)B_{n}\ast G$ whose kernel is zero.

By the isomorphism theorems: $B_{n}/L_{i}\cong $ $B_{n}/ZB_{n}/L_{i}/ZB_{n}%
\cong $ $C_{n}$ /$\sqrt{\overline{m}_{i}C_{n}}$ and $B_{n}/n_{0}\cong
B_{n}/ZB_{n}/n_{0}/ZB_{n}\cong C_{n}/\overline{n}_{0}\cong C.$

Therefore: $B_{n}\ast G/n_{0}\ast G\cong CG$ and for each $i\neq $0, $%
S_{i}\ast G=B_{n}\ast G/L_{i}\ast G\cong C_{n}/\sqrt{\overline{m}_{i}C_{n}}%
\ast G$ is a simple finite dimensional algebra, as we proved in Proposition
?.

Then we have an injective ring homomorphism:

$\psi :B_{n}\ast G\rightarrow \underset{i\in I-\{0\}}{\dprod }S_{i}\ast
G\times CG\times \underset{c\in C-\{0\}}{\dprod }A_{n}\ast G$

By the simplicity of $S_{i}\ast G$ and $A_{n}\ast G$ for the idempotent $e$
we have: $S_{i}\ast GeS_{i}\ast G=S_{i}\ast G$ and $A_{n}\ast GeA_{n}\ast
G=A_{n}\ast G$.

$\psi (e)=\overset{\wedge }{e}=((e),e,(e))$ is an idempotent of $\underset{%
i\in I-\{0\}}{\dprod }S_{i}\ast G\times CG\times \underset{c\in C-\{0\}}{%
\dprod }A_{n}\ast G$ such that:

$\underset{i\in I-\{0\}}{(\dprod }$S$_{i}$*G$\times $CG$\times \underset{%
c\in C-\{0\}}{\dprod }$A$_{n}$*G)$\overset{\wedge }{e}\underset{i\in I-\{0\}}%
{(\dprod }$S$_{i}$*G$\times $CG$\times \underset{c\in C-\{0\}}{\dprod }$A$%
_{n}$*G)$=$

$\underset{i\in I-\{0\}}{\text{(}\dprod }$S$_{i}$*G$\times $CGe$\times 
\underset{c\in C-\{0\}}{\dprod }$A$_{n}$*G and as in Therorem ? there is an
injective ring homomorphism:

$B_{n}\ast G/B_{n}\ast GeB_{n}\ast G\rightarrow CG/CGe.$

It follows $B_{n}\ast G/B_{n}\ast GeB_{n}\ast G$ is a finite dimensional $C$%
-algebra.
\end{proof}

We have all the ingredients to prove for the homogenized Weyl algebras,
theorem analogous to Theorem9 and Theorem 10, what was essential in the
proof of those theorems was the fact $C_{n}\ast G/C_{n}\ast GeC_{n}\ast G$
is a finite dimensional $C$-algebra and the fact $C_{n}$ is a finitely
generated $C_{n}^{G}$-module. Then the proof of the next two theorems
follows by similar arguments to those used in the proof of Theorem 9 and
Theorem 10 and we will skip it.

\begin{theorem}
Let $B_{n}$ =C\TEXTsymbol{<}X$_{1}$,X$_{2}$,...X$_{n}$,Y$_{1}$,Y$_{2}$,...Y$%
_{n}$,Z\TEXTsymbol{>}/\{[X$_{i}$,X$_{j}$],[Y$_{i}$,Y$_{j}$],[X$_{i}$,Y$_{j}$%
]-$\delta _{ij}$Z$^{2}$,[X$_{i}$,Z],[Y$_{i}$,Z]\} be the homogenized Weyl
algebra in 2n+1 variables, $G$ a finite group of grade preserving
automorphisms of $B_{n}$ and $e$ the idempotent of $B_{n}\ast G$, $%
e=1/\left\vert G\right\vert \underset{\sigma \in G}{\sum }\sigma $.
Moreover, assume $B_{n}$ satisfies the following conditions:

i) For all $\sigma \in G$, $\sigma (Z)=Z.$

ii) The $C$-subspace $V=\overset{n}{\underset{i=1}{\oplus }}CX_{i}\oplus 
\overset{n}{\underset{j=1}{\oplus }}CY_{j}$ of $B_{n}$ is $G$- invariant.

iii) If $v\in V$, $v\neq 0$ and $\sigma (v)=v$, then $\sigma =1.$

Then: writing $R=B_{n}\ast G$ and $T=eB_{n}\ast Ge\cong B_{n}^{G}$ and
denoting by \textsl{S}$_{R}$ and \textsl{S}$_{T}$ the categories of finite
dimensional $R$ and $T$-modules, respectively and by $\func{mod}_{R}$, $%
\func{mod}_{T}$, the categories of finitely generated, $R$ and $T$ -modules,
respectively. Then there is a commutative diagram of categories and functors:

$%
\begin{array}{ccc}
\func{mod}_{R} & \overset{Hom_{R}(\func{Re},-)}{\rightarrow } & \func{mod}%
_{T} \\ 
\downarrow \pi _{R} &  & \downarrow \pi _{T} \\ 
\func{mod}_{R}/\QTR{sl}{S}_{R} & \overset{H}{\rightarrow } & \func{mod}_{T}/%
\QTR{sl}{S}_{T}%
\end{array}%
$, with $H$ an equivalence.
\end{theorem}

We have also the graded version:

\begin{theorem}
Let $B_{n}$ =C\TEXTsymbol{<}X$_{1}$,X$_{2}$,...X$_{n}$,Y$_{1}$,Y$_{2}$,...Y$%
_{n}$,Z\TEXTsymbol{>}/\{[X$_{i}$,X$_{j}$],[Y$_{i}$,Y$_{j}$],[X$_{i}$,Y$_{j}$%
]-$\delta _{ij}$Z$^{2}$,[X$_{i}$,Z],[Y$_{i}$,Z]\} be the homogenized Weyl
algebra in 2n+1 variables, $G$ a finite group of grade preserving
automorphisms of $B_{n}$ and $e$ the idempotent of $B_{n}\ast G$, $%
e=1/\left\vert G\right\vert \underset{\sigma \in G}{\sum }\sigma $.
Moreover, assume $B_{n}$ satisfies the following conditions:

i) For all $\sigma \in G$, $\sigma (Z)=Z.$

ii) The $C$-subspace $V=\overset{n}{\underset{i=1}{\oplus }}CX_{i}\oplus 
\overset{n}{\underset{j=1}{\oplus }}CY_{j}$ of $B_{n}$ is $G$- invariant.

iii) If $v\in V$, $v\neq 0$ and $\sigma (v)=v$, then $\sigma =1.$ Then:
writing $R=B_{n}\ast G$ and $T=eB_{n}\ast Ge\cong B_{n}^{G}$ and considering
both as positively graded algebras, denoting by \textsl{S}$_{R}$ and \textsl{%
S}$_{T}$ the categories of finite dimensional graded $R$ and $T$-modules,
respectively and by $g_{R}$, $gr_{T}$, the categories of finitely generated, 
$R$ and $T$ -modules, respectively and by $tails_{R}$ and $tails_{T}$ the
quotient categories $gr_{R}/$\textsl{S}$_{R}$ and $gr_{T}/$ \textsl{S}$_{T}$%
. Then there is a commutative diagram of categories and functors: $%
\begin{array}{ccc}
gr_{R} & \overset{Hom_{R}(\func{Re},-)}{\rightarrow } & gr_{T} \\ 
\downarrow \pi _{R} &  & \downarrow \pi _{T} \\ 
tails_{R} & \overset{H}{\rightarrow } & tails_{T}%
\end{array}%
$, with $H$ an equivalence.
\end{theorem}

We also have as in Theorem 10:

Changing notation, from the graded version of the theorem we have an
equivalence of categories: $Qgr_{B_{n}\ast G}\cong Qgr_{B_{n}^{G}}.$ This
equivalence induces at the level of bounded derived equivalences: $\emph{D}%
^{b}($ $Qgr_{B_{n}\ast G})\cong \emph{D}^{b}(Qgr_{B_{n}^{G}}).$

As remarked above, $G$ acts as an automorphism group of $B_{n}^{!}$ , the
Yoneda algebra of $B_{n}$ and the Yoneda algebra of $B_{n}\ast G$ is $%
B_{n}^{!}$ $\ast G$.

The algebra $B_{n}^{!}$ $\ast G$ is Koszul selfinjective. By a theorem of $%
[MS$ and by $[MM]$ there is an equivalence of triangulated categories: 
\underline{$gr$}$_{B_{n}^{!}\ast G}\cong \emph{D}^{b}($ $Qgr_{B_{n}\ast G})$%
, where \underline{$gr$}$_{B_{n}^{!}\ast G}$ denotes the stable category of
finitely generated graded modules. Therefore: there is an equivalence of
triangulated categories: \underline{$gr$}$_{B_{n}^{!}\ast G}\cong \emph{D}%
^{b}(Qgr_{B_{n}^{G}})$.

We have proved the following:

\begin{corollary}
Let $B_{n}$ =C\TEXTsymbol{<}X$_{1}$,X$_{2}$,...X$_{n}$,Y$_{1}$,Y$_{2}$,...Y$%
_{n}$,Z\TEXTsymbol{>}/\{[X$_{i}$,X$_{j}$],[Y$_{i}$,Y$_{j}$],[X$_{i}$,Y$_{j}$%
]-$\delta _{ij}$Z$^{2}$,[X$_{i}$,Z],[Y$_{i}$,Z]\} be the homogenized Weyl
algebra in 2n+1 variables, $G$ a finite group of grade preserving
automorphisms of $B_{n}$ and $e$ the idempotent of $B_{n}\ast G$, $%
e=1/\left\vert G\right\vert \underset{\sigma \in G}{\sum }\sigma $.
Moreover, assume $B_{n}$ satisfies the following conditions:

i) For all $\sigma \in G$, $\sigma (Z)=Z.$

ii) The $C$-subspace $V=\overset{n}{\underset{i=1}{\oplus }}CX_{i}\oplus 
\overset{n}{\underset{j=1}{\oplus }}CY_{j}$ of $B_{n}$ is $G$- invariant.

iii) If $v\in V$, $v\neq 0$ and $\sigma (v)=v$, then $\sigma =1.$

Let $B_{n}^{!}$ be the Yoneda algebra of $B_{n}$. Then there are
isomorphisms of triangulated categories: \underline{$gr$}$_{B_{n}^{!}\ast
G}\cong \emph{D}^{b}(Qgr_{B_{n}^{G}})\cong \emph{D}^{b}($ $Qgr_{B_{n}\ast
G}).$In particular, the categories $\emph{D}^{b}(Qgr_{B_{n}^{G}})$ and $%
\emph{D}^{b}($ $Qgr_{B_{n}\ast G})$ have Auslander-Reiten triangles and they
are of the form $ZA_{n}.$
\end{corollary}

\begin{proof}
As before, the proof uses the fact $B_{n}^{!}\ast G$ is selfinjective Koszul
and results from [MZ] .
\end{proof}

\begin{center}
{\LARGE References}

\bigskip
\end{center}

[AR] Auslander, M.; Reiten, I.; McKay quivers and extended Dynkin diagrams.
Trans. Amer. Math. Soc. 293 (1986), no. 1, 293--301

[Be] Benson D,; Polynomial Invariants of Finite Groups. London Mathematical
Society Lecture Notes Series 190, Cambridge University Press, 1993.

[CR] Curtis Ch. W. Reiner I. Representation Theory of Finite Groups and
Associative Algebras, Pure and Applied Mathematics Vol XI, Interscience 1962

[C-BH] Crawley-Boevey, W.; Holland, M. P. Noncommutative deformations of
Kleinian singularities. Duke Math. J. 92 (1998), no. 3, 605--635.

[C-B] Crawley-Boevey W. DMV\ Lectures on Representations of quivers,
preprojective algebras and deformation of quotient singularities.(preprint
Leeds)

[Co] Coutinho S.C. A Primer of Algebraic D-modules, London Mathematical
Society, Students Texts 33, 1995

[Ga] Gabriel, Pierre Des cat\'{e}gories ab\'{e}liennes. Bull. Soc. Math.
France 90 1962 323--448.

[GH]\ Green, E.; Huang, R. Q. Projective resolutions of straightening closed
algebras generated by minors. Adv. Math. 110 (1995), no. 2, 314--333.

[GM1] Green, E. L.; Mart\'{\i}nez Villa, R.; Koszul and Yoneda algebras.
Representation theory of algebras (Cocoyoc, 1994), 247--297, CMS Conf.
Proc., 18, Amer. Math. Soc., Providence, RI, 1996.

[GM2] Green, E. L.; Mart\'{\i}nez-Villa, R.; Koszul and Yoneda algebras. II.
Algebras and modules, II (Geiranger, 1996), 227--244, CMS Conf. Proc., 24,
Amer. Math. Soc., Providence, RI, 1998.

[GMT]\ Guo, Jin Yun; Mart\'{\i}nez-Villa, R.; Takane, M.; Koszul generalized
Auslander regular algebras. Algebras and modules, II (Geiranger, 1996),
263--283, CMS Conf. Proc., 24, Amer. Math. Soc., Providence, RI, 1998

[GuM] Guo, Jin Yun; Mart\'{\i}nez-Villa, R.; Algebra pairs associated to
McKay quivers. Comm. Algebra 30 (2002), no. 2, 1017--1032.

[Ku] Kunz E.; Introduction to Commutative Algebra and Algebraic Geometry,
Birkh\"{a}user, 1985

[Le] Lenzing H.;, Polyhedral groups and the geometric study of tame
hereditary algebras (preprint Paderborn).

[Li] Li Huishi. Noncommutative Groebner Bases and Filtered-Graded Transfer,
Lecture Notes in Mathematics 1795, Springer, 2002

[MM] Mart\'{\i}nez-Villa, R., Martsinkovsky; A. Stable Projective Homotopy
Theory of Modules, Tails, and Koszul Duality, (aceptado 11 de septiembre
2009), Comm. Algebra 38 (2010), no. 10, 3941--3973.

[MS] Mart\'{\i}nez Villa, Roberto; Saor\'{\i}n, M.; Koszul equivalences and
dualities. Pacific J. Math. 214 (2004), no. 2, 359--378.

[MZ] Mart\'{\i}nez-Villa, Roberto; Zacharia, Dan; Approximations with
modules having linear resolutions. J. Algebra 266 (2003), no. 2, 671--697.

[MV1] Mart\'{\i}nez-Villa, R.; Skew group algebras and their Yoneda
algebras. Math. J. Okayama Univ. 43 (2001), 1--16.

[MV2] Mart\'{\i}nez-Villa, R.; Serre duality for generalized Auslander
regular algebras. Trends in the representation theory of finite-dimensional
algebras (Seattle, WA, 1997), 237--263, Contemp. Math., 229, Amer. Math.
Soc., Providence, RI, 1998.

[MV3] Mart\'{\i}nez-Villa, R.; Applications of Koszul algebras: the
preprojective algebra. Representation theory of algebras (Cocoyoc, 1994),
487--504, CMS Conf. Proc., 18, Amer. Math. Soc., Providence, RI, 1996.

[MV4] Martinez-Villa, R. Graded, Selfinjective, and Koszul Algebras, J.
Algebra 215, 34-72 1999

[MMo] Martinez-Villa, R. Mondragon J. On the homogeneized Weyl Algebra.
(preprint 2011).

[Mi] Miyachi, Jun-Ichi; Derived Categories with Applications to
Representations of Algebras, Chiba Lectures, 2002.

[Mu] M\"{u}ller W. Darstellungstheorie von eindlichen Gruppen, Teubner
Studienb\"{u}cher, 1980

[Mc] McKay J.;, Graphs, singularities and finite groups, Proc. Sympos. Pure
Math., vol. 37, Amer. Math. Soc., Providence, R. I., 1980, pp. 183-186.

[P] Popescu N.; Abelian Categories with Applications to Rings and Modules;,
L.M.S. Monographs 3, Academic Press 1973.

[RR] Reiten, I.; Riedtmann, Ch. Skew group algebras in the representation
theory of Artin algebras. J. Algebra 92 (1985), no. 1, 224--282.

[Sm] Smith, P.S.; Some finite dimensional algebras related to elliptic
curves, Rep. Theory of Algebras and Related Topics, CMS Conference
Proceedings, Vol. 19, 315-348, Amer. Math. Soc 1996.

[Ste] Steinberg R.,; Finite subgroups of SU$_{2}$, Dynkin Diagrams and
Affine Coxeter elements. Pacific Journal of Mathematics, Vol. 118, No. 2,
1985.

[Stu] Sturmfels B.; Algorithms in Invariant Theory, Texts and Monographs in
Symbolic Computation, Springer-Verlag, 1993

\end{document}